\documentclass[10pt]{article}

\usepackage{amssymb,amsmath,amsfonts,amsthm}
\usepackage{graphicx}
\usepackage[a4paper,includeheadfoot,left=4cm,right=4cm,top=2.8cm,bottom=2.8cm]{geometry}

\usepackage{color}
\definecolor{marin}{rgb}{0.,0.3,0.7}
\usepackage[colorlinks,citecolor=marin,linkcolor=marin,urlcolor=marin,bookmarksopen,bookmarksnumbered]{hyperref}

\usepackage[numbers]{natbib}
\makeatletter
\let\@fnsymbol\@arabic
\makeatother

\newcommand{\al}{\alpha}
\newcommand{\be}{\beta}
\newcommand{\ga}{\gamma}
\newcommand{\de}{\delta}
\newcommand{\ep}{\varepsilon}
\newcommand{\et}{\eta}
\newcommand{\la}{\lambda}
\newcommand{\rh}{\rho}
\newcommand{\si}{\sigma}
\newcommand{\ta}{\tau}
\newcommand{\om}{\omega}
\newcommand{\omt}{\varpi}
\newcommand{\Ga}{\Gamma}

\newcommand{\Ph}{\Phi}
\newcommand{\Om}{\Omega}
\newcommand{\N}{\mathbb{N}}
\newcommand{\Z}{\mathbb{Z}}

\newcommand{\R}{\mathbb{R}}
\newcommand{\C}{\mathbb{C}}
\newcommand{\T}{\mathbb{T}}
\newcommand{\Fc}{\mathcal{F}}
\newcommand{\Ic}{\mathcal{I}}
\newcommand{\Nc}{\mathcal{N}}
\newcommand{\Pc}{\mathcal{P}}
\newcommand{\Qc}{\mathcal{Q}}
\newcommand{\abf}{\mathbf{a}}
\newcommand{\bbf}{\mathbf{b}}
\newcommand{\dbf}{\mathbf{d}}
\newcommand{\ebf}{\mathbf{e}}
\newcommand{\kbf}{\mathbf{k}}
\newcommand{\lbf}{\mathbf{l}}
\newcommand{\vbf}{\mathbf{v}}
\newcommand{\zbf}{\mathbf{z}}
\newcommand{\Fbf}{\mathbf{F}}
\newcommand{\Hbf}{\mathbf{H}}
\newcommand{\Pbf}{\mathbf{P}}
\newcommand{\Sbf}{\mathbf{S}}

\newcommand{\Ombf}{\boldsymbol{\Om}}
\newcommand{\Gabf}{\boldsymbol{\Ga}}
\newcommand{\omtbf}{\boldsymbol{\omt}}
\providecommand{\abs}[1]{\lvert#1\rvert}
\providecommand{\absbig}[1]{\bigl\lvert#1\bigr\rvert}
\providecommand{\absBig}[1]{\Bigl\lvert#1\Bigr\rvert}
\providecommand{\absbigg}[1]{\biggl\lvert#1\biggr\rvert}
\providecommand{\norm}[1]{\lVert#1\rVert}
\providecommand{\normbig}[1]{\bigl\lVert#1\bigr\rVert}
\providecommand{\normBig}[1]{\Bigl\lVert#1\Bigr\rVert}
\providecommand{\normv}[1]{\ensuremath{{\lVert\hskip-1pt\lvert}#1{\rvert\hskip-1pt\rVert}}}

\providecommand{\normvBig}[1]{\ensuremath{{\Bigl\lVert\hskip-1pt\Bigl\lvert}#1{\Bigr\rvert\hskip-1pt\Bigr\rVert}}}
\providecommand{\skla}[1]{\langle#1\rangle}
\newcommand{\formulatext}[1]{\qquad\text{#1}\qquad}
\newcommand\myfor{\formulatext{for}}
\renewcommand\forall{\formulatext{for all}}
\newcommand\myand{\formulatext{and}}

\newcommand\with{\formulatext{with}}
\newcommand{\sfrac}[2]{\mbox{$\textstyle\frac{#1}{#2}$}}
\newcommand{\iu}{\mathrm{i}}
\newcommand{\e}{\mathrm{e}}
\newcommand{\dd}{\mathrm{d}}
\DeclareMathOperator{\ReT}{Re}
\DeclareMathOperator{\ImT}{Im}
\DeclareMathOperator{\diag}{diag}
\DeclareMathOperator{\sgn}{sgn}
\DeclareMathOperator{\argmax}{argmax}
\newtheorem{theorem}{Theorem}[section]
\newtheorem{lemma}[theorem]{Lemma}
\newtheorem{proposition}[theorem]{Proposition}

\theoremstyle{definition}
\newtheorem{assum}{Assumption}
\newtheorem{remark}[theorem]{Remark}

\newcommand{\jvec}{{\skla{j}}}
\newcommand{\disc}{\mathcal{K}}
\newcommand{\ind}{\mathcal{Z}}
\newcommand{\res}{\mathcal{M}}
\newcommand{\Sch}{Schr\"o\-ding\-er }

\title{Plane wave stability of the split-step Fourier method for the nonlinear Schr\"odinger equation}

\author{Erwan Faou\,\thanks{INRIA and ENS Cachan Bretagne, 
          Avenue Robert Schumann, 
          F-35170 Bruz, France
          ({\tt Erwan.Faou@inria.fr}).}
         \textsuperscript{,}\thanks{D\'{e}partement de math\'{e}matiques et applications, 
          \'{E}cole normale sup\'{e}rieure,
          45 rue d'Ulm,
          F-75230 Paris Cedex 05, France.}
        \and
        Ludwig Gauckler\,\thanks{Institut f\"ur Mathematik,
          Technische Universit\"at Berlin,
          Stra{\ss}e des 17.\ Juni 136,
          D-10623 Berlin, Germany
          ({\tt gauckler@math.tu-berlin.de}).}
        \and
        Christian Lubich\,\thanks{Mathematisches Institut,
          Universit\"at T\"ubingen,
          Auf der Morgenstelle 10,
          D-72076 T\"ubingen, Germany
          ({\tt lubich@na.uni-tuebingen.de}).}
}

\date{Version of 2 Dezember 2013}

\begin{document}

\maketitle

\begin{abstract}
Plane wave solutions to the cubic nonlinear Schr\"odinger equation on a torus have recently been shown to behave orbitally stable. Under generic perturbations of the initial data that are small in a high-order Sobolev norm, plane waves are stable over long times that extend to arbitrary negative powers of the smallness parameter. The present paper studies the  question as to whether numerical discretizations by the split-step Fourier method inherit such a generic long-time stability property. This can indeed be shown under a condition of linear stability and a non-resonance condition. They can both be verified if the time step-size is restricted by a CFL condition in the case of a constant plane wave. The proof first uses a Hamiltonian reduction and transformation and then modulated Fourier expansions in time. It provides detailed insight into the structure of the numerical solution.\\[1.5ex]
\textbf{Mathematics Subject Classification (2010):} Primary 65P10, 
65P40; 
secondary: 65M70.
\end{abstract}

\section{Introduction}

We consider the \emph{cubic nonlinear Schr\"odinger} equation
\begin{equation}\label{eq-nls}
\iu \frac{\partial}{\partial t} u = -\Delta u + \lambda\abs{u}^2u, \qquad u=u(x,t)
\end{equation}
in the defocusing ($\la=+1$) or focusing case ($\la=-1$). 
We impose periodic boundary conditions in arbitrary spatial dimension $d\ge 1$: the spatial variable $x$ belongs to the $d$-dimensional torus $\T^d=\R^d/(2\pi\Z)^d$. 

This nonlinear \Sch equation has a class of simple solutions, 
the \emph{plane wave solutions}
\begin{equation}\label{eq-pw}
u(x,t)=\rho\, \e^{\iu(\ell\cdot x - \om t)} 
\end{equation}
for $\rho\ge 0$, $\ell\in\Z^d$ and $\om= \abs{\ell}^2 + \la\rho^2$, where $\ell\cdot x=\ell_1x_1+\dots+\ell_dx_d$ and $\abs{\ell}^2=\ell\cdot \ell$. A natural question is whether these plane wave solutions~\eqref{eq-pw} are stable under small perturbations of the initial value. In this context it is common knowledge that a \emph{linear stability analysis}, where one examines the eigenvalues of the linearization of the nonlinear \Sch equation~\eqref{eq-nls} around a plane wave, leads to the condition $1+2\la\rh^2\ge 0$ for (linear) stability, see for instance \cite[Sect.~5.1.1]{Agrawal2006}. Since nonlinear effects are ignored, the validity of such a linear stability analysis is inherently restricted to a short time interval. Stability and instability on long time intervals of plane waves in the exact solution is discussed in the recent papers \cite{Faoua} and \cite{Hani}, respectively. Of particular importance for the present paper is \cite{Faoua}, where orbital stability over long times is shown for perturbations in high-order Sobolev spaces. Orbital stability means that the solution stays close to the orbit \eqref{eq-pw}.

From the viewpoint of numerical analysis, it is  of interest  whether (and if so why) a numerical method shares the stability or instability of the exact solution near plane waves. This is the topic of the present paper. 

This problem can be traced back to the seminal paper \cite{Weideman1986} by Weideman \& Herbst from 1986. In that paper, conditions on the discretization parameters for various numerical methods are derived that ensure that the numerical solution shares the \emph{linear stability} of the exact solution. This is done by examining the eigenvalues of the linearization around a plane wave of a numerical method applied to~\eqref{eq-nls}. Such a linear stability analysis has recently been extended to different numerical methods \cite{Cano2013,Dahlby2009,Khanamiryan,Lakoba2013}.

In the present paper, we take up this line of research. In contrast to previous work \cite{Cano2013,Dahlby2009,Khanamiryan,Lakoba2013,Weideman1986}, however, we are interested in the \emph{long-time behaviour} of a numerical method near plane waves, and hence a linear stability analysis is of limited use. We pursue the question as to whether the remarkably stable behaviour on long time intervals of the exact solution near plane waves~\cite{Faoua} is shared by one of the most popular numerical methods for the nonlinear \Sch equation, the \emph{split-step Fourier method}~\cite{Hardin1973}. This method combines a Fourier collocation in space with a Strang splitting in time, see Subsect.~\ref{subsec-method}. It integrates plane wave solutions~\eqref{eq-pw} exactly. Our main result states that the long-time orbital stability of the exact solution near plane waves transfers to the numerical solution, see Subsect.~\ref{subsec-mainresult} for a precise statement. In the case of a spatially constant plane wave ($\ell=0$ in \eqref{eq-pw}), the case considered by Weideman \& Herbst \cite{Weideman1986}, it is further shown that the assumptions of this main result essentially hold under a CFL condition on the discretization parameters, see Subsect.~\ref{subsec-discassum}.

The long-time stability result of the present paper deals with the \emph{completely resonant} equation~\eqref{eq-nls}: the eigenvalues (frequencies) of the linear part of the equation are $\abs{j}^2$, $j\in\Z^d$, whose integer linear combinations may vanish identically. This is in marked contrast to previous long-time stability results for numerical discretizations of nonlinear Hamiltonian partial differential equations that consider non-resonant situations. See \cite{Faou2009a,Faou2009b,GaucklerDiss,Gauckler2010b} for the split-step Fourier method applied to the nonlinear \Sch equation, where \eqref{eq-nls} is considered with an additional (generic but artificial) convolution term $V\star u$ in order to have non-resonant frequencies. Another feature of the result in the present paper is that it covers a much larger class of \emph{initial values that are not small} than the aforementioned previous stability results that all deal with small initial values.
 
The proof of our stability result is given in Sects.~\ref{sec-trf}--\ref{sec-nonres}. We first eliminate, in Sect.~\ref{sec-trf}, the principal Fourier mode from the numerical scheme with a sequence of transformations and reductions. The resulting system of equations has small initial values. This enables us to use the technique of modulated Fourier expansions for its long-time analysis, see Sect.~\ref{sec-mfe}. It is likely that normal form techniques in the spirit of \cite{Faou2009a,Faou2009b} would lead to similar conclusions, but we have not worked out the details. In order to obtain results that are valid on long time intervals, the frequencies have to satisfy a certain non-resonance condition. In fact, the completely resonant frequencies of the nonlinear \Sch equation are modified during the transformations of Sect.~\ref{sec-trf}, and we are able to verify a non-resonance condition for the new frequencies in the final Sect.~\ref{sec-nonres}.

\section{Numerical method and statement of the main results}

\subsection{The split-step Fourier method}\label{subsec-method}

We discretize the nonlinear \Sch equation \eqref{eq-nls} with the split-step Fourier meth\-od as introduced in \cite{Hardin1973,Taha1984,Weideman1986}. In this method, the equation is discretized in space by a spectral collocation method and in time by a splitting integrator. 

\medskip\noindent{\bf Discretization in space.} 
For the discretization in space we make the ansatz
\[
u_K(x,t) = \sum_{j\in\disc} u_j(t) \e^{\iu (j\cdot x)} \with \disc := \{-K,\dots,K-1\}^d
\]
with the spatial discretization parameter $K$. For fixed $t$, $u_K(\cdot,t)$ is a \emph{trigonometric polynomial} which is uniquely determined by its values in the collocation points $x_j = \pi j/K$, $j\in\disc$. Requiring that the ansatz $u_K$ fulfills the nonlinear \Sch equation \eqref{eq-nls} in the collocation points leads to the equation
\begin{equation}\label{eq-nlssemi}
\iu \frac{\partial}{\partial t} u_K = -\Delta u_K + \lambda\Qc\bigl(\abs{u_K}^2u_K\bigr), \qquad u_K(\cdot,0) = \Qc(u(\cdot,0)),
\end{equation}
where the trigonometric interpolation $\Qc(u)$ (with respect to the spatial variable $x$) of a function $u(x) = \sum_{k\in\Z^d}u_k \e^{\iu (k\cdot x)}$  is the uniquely determined trigonometric polynomial that interpolates $u$ in the collocation points. This trigonometric interpolation is given by
\[
\Qc(u) = \sum_{j\in\disc} \widetilde{u}_j \e^{\iu (j\cdot x)} \with \widetilde{u}_j = \sum_{k\in\Z^d : k\equiv j \bmod{2K}} u_k,
\]
where the congruence modulo $2K$ has to be understood entrywise. 

\medskip\noindent{\bf Discretization in time.} 
Equation \eqref{eq-nlssemi} is then discretized in time by a splitting integrator with time step-size $h$. For this purpose we split \eqref{eq-nlssemi} in its linear and its nonlinear part,
\[
\iu \frac{\partial}{\partial t} u_K = -\Delta u_K \myand  \iu \frac{\partial}{\partial t} u_K = \lambda\Qc\bigl(\abs{u_K}^2u_K\bigr).
\]
Denoting by $\Ph_{\text{linear}}^h$ and $\Ph_{\text{nonlinear}}^h$ the flows over a time $h$ of these equations, we compute approximations $u_K^{n}$ to $u_K(\cdot,t_{n})$ at discrete times $t_{n}=nh$ by 
\begin{subequations}\label{eq-splitting}
\begin{equation}\label{eq-splitstep}
u_K^{n+1} = \Ph_{\text{linear}}^h \circ \Ph_{\text{nonlinear}}^h (u_K^n).
\end{equation}
The initial value $u_K^0$ is chosen as
\begin{equation}\label{eq-splitinit}
u_K^0 = u_K(\cdot,0) = \Qc(u(\cdot,0)).
\end{equation}
\end{subequations}
Equations \eqref{eq-splitting} provide a fully discrete scheme for the numerical solution of the nonlinear \Sch equation \eqref{eq-nls}, the \emph{split-step Fourier method}. Since its introduction in \cite{Hardin1973} it has become a widely used and well analysed method, see for example \cite{Agrawal2006,Faou2012,GaucklerDiss,Jin2011,Taha1984,Weideman1986} and references therein.

\medskip\noindent{\bf Computational aspects.} 
In~\eqref{eq-splitting}, both flows $\Ph_{\text{linear}}^h$ and $\Ph_{\text{nonlinear}}^h$ can be computed exactly in an efficient way. The flow of the linear equation is given in terms of the Fourier coefficients $u_j$ of a trigonometric polynomial $u(x) = \sum_{j\in\disc}u_j\e^{\iu(j\cdot x)}$ by
\begin{equation}\label{eq-flow-linear}
\Ph_{\text{linear}}^h(u) = \sum_{j\in\disc}\e^{-\iu\abs{j}^2h} u_j\e^{\iu(j\cdot x)}.
\end{equation}
Thus, it can be computed easily in terms of these Fourier coefficients. On the other hand, the flow of the nonlinear equation is given by
\begin{equation}\label{eq-flow-nonlinear}
\Ph_{\text{nonlinear}}^h(u) = \Qc\Bigl( \e^{-\iu\la\abs{u}^2h} u\Bigr),
\end{equation}
i.e., $\Ph_{\text{nonlinear}}^h(u) (x_j) = \e^{-\iu\la\abs{u(x_j)}^2h} u(x_j)$ for all $j\in\disc$. This is easy to compute in terms of the function values in the collocation points. Note that the fast Fourier transform provides an efficient tool to switch from Fourier coefficients to function values in the collocation points and vice-versa. The computational cost per time step is thus of order  $K^d\log K^d$.

\medskip\noindent{\bf Plane waves in the split-step Fourier method.}  
The split-step Fourier method \eqref{eq-splitting} has \emph{plane wave solutions}
\begin{equation}\label{eq-pwnum}
 u^n_K(x) = \rho \e^{\iu(\ell\cdot x - \om t_n)} \myfor u_K^0(x)= \rho \e^{\iu (\ell\cdot x)}  
\end{equation}
with $\om= \abs{\ell}^2 + \la\rho^2$ and $\ell\in\disc$. 
In other words, the plane wave solutions $\rho \e^{\iu(\ell\cdot x - \om t)}$ \eqref{eq-pw} of the nonlinear \Sch equation \eqref{eq-nls} are integrated exactly by the split-step Fourier method if $\ell\in\disc$. It is the stability of these plane wave solutions \eqref{eq-pwnum} under perturbations of the initial value that we are interested in.

\subsection{Long-time orbital stability}\label{subsec-mainresult}

For the study of the stability of plane wave solutions \eqref{eq-pwnum}, with fixed vector $\ell\in\disc$, we impose the following assumptions (with constants that do not depend on the discretization parameters $h$ and $K$). 

\begin{assum}\label{assum-linearstability}
We assume that the time step-size $h$ and the spatial discretization parameter $K$ fulfill together with $\rh\ge 0$ (which will be chosen later as the $L_2$-norm of the initial value)
\begin{equation}\label{eq-linearstability}
\bigl( \cos(n(j) h) - h\la\rh^2\sin(n(j) h) \bigr)^2 \le 1 - c_1 h^2 \forall j\in\ind := \disc\setminus\{0\}
\end{equation}
with a positive constant $c_1$, where
\begin{equation}\label{eq-nj}
n(j) = \sfrac12 \abs{\ell+j\bmod{2K}}^2 + \sfrac12 \abs{\ell-j\bmod{2K}}^2 - \abs{\ell}^2.
\end{equation}
\end{assum}

Assumption~\ref{assum-linearstability} ensures that the \emph{frequencies}
\begin{equation}\label{eq-freq}
\om_j = \sfrac12 \abs{\ell+j\bmod{2K}}^2 - \sfrac12 \abs{\ell-j\bmod{2K}}^2 + \frac{\arccos\bigl(\cos(n(j) h) - h\la\rh^2\sin(n(j) h)\bigr)}{h \sgn\bigl(\sin(n(j) h) +h\la\rh^2\cos(n(j) h)\bigr)}
\end{equation}
are well defined for all $j\in\ind$.
These frequencies show up after a linearization of the split-step Fourier method around a plane wave.  This linearization has eigenvalues $\e^{-\iu\omega_j h}$, see Sect.~\ref{sec-trf}. 

As a second assumption we need a \emph{non-resonance condition}. Ideally we would like to impose this condition directly on the frequencies $\om_j$. For the verification of the non-resonance condition, however, it turns out to be appropriate to consider modifications of these frequencies.

\begin{assum}\label{assum-nonres}
We assume that the time step-size $h$, the spatial discretization parameter $K$ and $\rh\ge 0$ are chosen such that there exist  \emph{modified frequencies} $\omt_j$, $j\in\ind$, with the following properties for some $N\ge 2$:

(a) The modified frequencies are close to the frequencies $\om_j$ \eqref{eq-freq}, 
\[
\abs{\omt_j - \om_j} \le \widehat{\ep} \forall j\in\ind
\]
with a small parameter $\widehat{\ep}$.

(b) There exist positive constants $c_2$, $\de_2$ and $s_2$ such that the following holds for all vectors $(k_j)_{j\in\ind}\in\Z^\ind$ of integers with $0<\sum_{j\in\ind} \abs{k_j} \le N+1$ and with $k_j\ne 0$ only if $k_l=0$ for all indices $l\ne j$ with $\omt_l = \omt_j$: 
if
\[
\de:= \absbigg{\frac{\e^{\iu(\sum_{j\in\ind} k_j\omt_j )h}-1}{h}} \le \de_2,
\]
then for all $l\in\ind$ satisfying $k_l\ne 0 $,
\[
\frac{\abs{l}^4}{\prod_{j\in\ind} \abs{j}^{2\abs{k_j}}} \le c_2 \de^{N/s_2}.
\]

(c) Complete resonances among the modified frequencies, i.e., $h \sum_{j\in\ind} k_j\omt_j \in 2\pi\Z$ for a vector $(k_j)_{j\in\ind}$ of integers with $\sum_{j\in\ind} \abs{k_j}\le N+1$, can only occur if 
\[
\sum_{\substack{j\in\ind\\ n(j) = m}} k_j =0 \forall m\in\Z.
\]
\end{assum}

Under these assumptions we will prove the following main result. Here we denote, for a trigonometric polynomial $u(x)=\sum_{j\in\disc}u_j\e^{\iu (j\cdot x)}$, by 
\[
\norm{u}_s^2 = \abs{u_0}^2 + \sum_{j\in\disc} \abs{j}^{2s} \abs{u_j}^2
\]
its Sobolev $H^s$-norm. We further denote by 
\[
\Fc_{\neg \ell}(u) = \Qc\bigl(\e^{-\iu (\ell\cdot x)} u - u_\ell\bigr) = \sum_{0\ne j\in\disc} u_{j+\ell \bmod{2K}} \e^{\iu (j\cdot x)}
\]
the same function with the $\ell$th Fourier coefficient set to zero, followed by a shift of Fourier coefficients by $\ell$ and a trigonometric interpolation.  Note that 
$\norm{ \Fc_{\neg \ell}(u) }_s$ measures the size of those Fourier coefficients whose subscript differs from $\ell$ modulo~$2K$.

\begin{theorem}\label{thm-main}
Fix an index $\ell\in\disc$, an integer $N\ge 2$ and positive numbers $c_1$,  $c_2$,  $s_2$, $\de_2$ and $\rh_1$. There exist $s_0$ and $C$ such that for every $s\ge s_0$ there exists $\ep_0>0$ such that the following holds: 
If the time step-size $h$, the spatial discretization parameter $K$ and $\rh\le \rh_1$  fulfill Assumptions~\ref{assum-linearstability} and \ref{assum-nonres} with some $\widehat{\ep}\le\ep_0$ (and with the prescribed constants $c_1$,  $c_2$,  $s_2$, $\de_2$), then for every initial value $u_K^0$ with
\[
\normbig{u_K^0}_0 = \rh \myand \normbig{ \Fc_{\neg \ell}(u_K^0) }_s \le \ep \le \ep_0
\]
we have the long-time stability estimate
\[
\normbig{\Fc_{\neg \ell}(u_K^n)}_s \le C \ep \myfor 0\le t_n = nh \le \max(\ep,\widehat{\ep})^{-N/2}.
\]
\end{theorem}

The proof of this theorem will be given in Sects.~\ref{sec-trf}--\ref{sec-mfe}. Theorem~\ref{thm-main} states that---under suitable assumptions---initial values that are close to a plane wave lead to numerical solutions that remain close to a plane wave for a long time, i.e., the numerical solution is concentrated in a single Fourier mode over long times. The closeness is measured by the Sobolev $H^s$-norm of $\Fc_{\neg \ell}(u)$.  This implies long-time \emph{orbital stability} in $H^s$, i.e., the numerical solution stays close to the orbit \eqref{eq-pwnum}, see \cite[Subsect.~3.4]{Faoua}.

The bounds $s_0$, $C$ and $\ep_0$ are independent of the discretization parameters $h$ and $K$ subject to Assumptions 1 and 2 and of the small parameters $\ep$ and $\widehat{\ep}$.  In more detail, the proof of Theorem~\ref{thm-main} shows that $s_0$ depends only on $d$ and $s_2$; $C$ depends only on $c_1$ and $\rh_1$; and  $\ep_0$ depends on $c_1$,  $c_2$, $d$, $\ell$, $N$, $s$, $s_2$, $\de_2$ and $\rh_1$.

\begin{remark}
The conclusion of Theorem~\ref{thm-main}  equally  holds if the (Lie-Trotter) splitting \eqref{eq-splitstep} is replaced by its symmetric version, the \emph{Strang splitting}
\[
u_K^{n+1} = \Ph_{\text{linear}}^{h/2} \circ \Ph_{\text{nonlinear}}^h \circ \Ph_{\text{linear}}^{h/2} (u_K^n).
\]
In fact, both numerical schemes differ only by half a time step with the linear flow at the beginning and at the end of the interval of integration. This does not affect the long-time stability. The same remark applies to the other version of the Strang splitting,
\[
u_K^{n+1} = \Ph_{\text{nonlinear}}^{h/2} \circ \Ph_{\text{linear}}^h \circ \Ph_{\text{nonlinear}}^{h/2} (u_K^n).
\]
\end{remark}

\subsection{Discussion of the assumptions}\label{subsec-discassum}

Assumptions~\ref{assum-linearstability} and \ref{assum-nonres}  for Theorem~\ref{thm-main}  exclude two different types of (potential) instabilities that show up on different time scales. Assumption~\ref{assum-linearstability}, which is derived in \cite[Sect.~5]{Weideman1986} and \cite{Lakoba2013} with a slightly different meaning of $\rh$, ensures that (numerical) plane wave solutions \eqref{eq-pwnum} are \emph{linearly stable}. This means that all eigenvalues of the linearization of the numerical scheme \eqref{eq-splitting} around a plane wave \eqref{eq-pwnum} are of modulus one. Eigenvalues of modulus larger than one would lead to an instability right from the start.  In contrast,  the non-resonance condition of Assumption~\ref{assum-nonres} on the frequencies is crucial for the proof of our \emph{long-time} result. Indeed, the longer the time interval under consideration is, the more the nonlinear interaction becomes relevant, possibly leading to resonance phenomena if the frequencies are resonant or close to resonant. 

A non-resonance condition as stated in Assumption~\ref{assum-nonres} is typically required in a long-time analysis of   Hamiltonian partial differential equations and their numerical discretizations, see for example \cite{Cohen2008,Faou2009a,Faou2009b,GaucklerDiss,Gauckler2010b}  for uses and discussions of similar conditions. Note, however, that we do not (and cannot) impose this non-resonance condition on the completely resonant frequencies of the nonlinear \Sch equation \eqref{eq-nls} and its discretization by the split-step Fourier method, but only on the frequencies $\om_j$  of \eqref{eq-freq} for  the linearization around a plane wave.

In Sect.~\ref{sec-nonres} we will prove the following theorem on a sufficient (though not necessary) condition under which Assumptions~\ref{assum-linearstability} and \ref{assum-nonres} hold in the case of a constant plane wave ($\ell=0$) for many values of $\rh=\norm{u_K^0}_0$ and $h$.

\begin{theorem}\label{thm-main2}
Let $\ell=0$, and fix $\rh_0> 0$ with
\begin{equation}\label{eq-linearstabilityexact}
1+2\la\rh_0^2>0,
\end{equation}
$h_0>0$ and $N\ge 2$. Then we have the following result.

For  every $\ga>0$  there exists a subset $\Pc(\ga)$ of $[0,\rh_0]\times[0,h_0]$ of Lebesgue measure $\abs{\Pc(\ga)}\ge \rh_0h_0-\ga$ such that Assumptions~\ref{assum-linearstability} and~\ref{assum-nonres} hold for all $(\rh,h)\in\Pc(\ga)$ and all $K$ that satisfy the restriction
\begin{equation}\label{eq-strongcfl}
dhK^2 + 2h\rh_0^2 \le \frac{\pi}{N+1}
\end{equation}
with small parameter $\widehat{\ep}=C_2 h^2$ and constants $c_1=c_1(\rh_0)$, $C_2=C_2(\rh_0)$, $c_2=c_2(h_0,N,\ga,\rh_0)$, $\de_2=1$ and $s_2=5^4N^5$. 
\end{theorem}

Theorem~\ref{thm-main} together with Theorem~\ref{thm-main2} is the discrete counterpart of \cite[Theorem~1.1]{Faoua}: for $\ell=0$, the long-time orbital stability of plane waves proven there for the exact solution transfers to the numerical discretization provided that the step-size restriction \eqref{eq-strongcfl} is fulfilled and that $\rh=\norm{u_K^0}_0$ and $h$ belong to a large set. In comparison with the result \cite[Theorem~1.1]{Faoua} for the exact solution, our discrete counterpart is valid on a time interval of length $\max(\ep,h^2)^{-N/2}$ instead of $\ep^{-N/2}$, and the value of $N$ is restricted by the step-size restriction~\eqref{eq-strongcfl}. These changes are due to the non-resonance condition that is much more involved for the numerical frequencies than for the analytical frequencies. 

Let us finally comment on condition~\eqref{eq-linearstabilityexact} in Theorem~\ref{thm-main2}. This condition ensures linear stability of (analytical) plane wave solutions \eqref{eq-pw} to the nonlinear Schr\"odinger equation \eqref{eq-nls}, i.e., that all eigenvalues of the linearization of the nonlinear Schr\"odinger equation around a plane wave \eqref{eq-pw} are real valued \cite{Weideman1986,Agrawal2006,Faoua}. 
Theorem~\ref{thm-main2} states in particular  that this implies linear stability of (numerical) plane wave solutions (\eqref{eq-linearstability} in Assumption~\ref{assum-linearstability}) under the step-size restriction~\eqref{eq-strongcfl}. Actually, weaker step-size restrictions that yield linear stability are discussed in detail in \cite[Sect.~5]{Weideman1986}, but \eqref{eq-strongcfl} is sufficient for our long-time result because it allows us to verify Assumption~\ref{assum-nonres}. On the other hand, \eqref{eq-linearstability} reduces to \eqref{eq-linearstabilityexact} (with $\rh_0=\rh$) in the limit $h\rightarrow 0$ for fixed $K$. 

For nonzero but small $\ell$, the condition of linear stability in Assumption~\ref{assum-linearstability} can still be expected to hold under a step-size restriction similar to \eqref{eq-strongcfl}. In this case, the frequencies $\om_j$ differ from those for $\ell=0$ only for large $j$ (we have $n(j)=\abs{j}^2$ and $\abs{\ell+j\bmod{2K}}=\abs{\ell-j\bmod{2K}}$ for all $j$ that are not large, and we have $c\abs{j}^2\le n(j)\le C\abs{j}^2$ for large $j$). This property can also be used to argue that the non-resonance condition of Assumption~\ref{assum-nonres} can be expected to hold for nonzero but small $\ell$, see Remark~\ref{rem-lne0} in Sect.~\ref{sec-nonres}. In one dimension ($d=1$), the linear stability of the split-step Fourier method for $\ell\ne 0$ has been recently analysed in detail by Lakoba \cite{Lakoba2013}.

\subsection{Numerical experiments}

We present numerical experiments which illustrate Theorem~\ref{thm-main} in situations that are not covered by Theorem~\ref{thm-main2}. They show in particular that the conditions of Theorem~\ref{thm-main2} are not necessary for the assumptions and conclusions of Theorem~\ref{thm-main} to hold.

Throughout, we let $\la=-1$ and $\rh^2 = 0.4$ such that $1+2\la\rh^2>0$, and hence we have linear stability of plane waves in the exact solution. We consider the nonlinear \Sch equation in dimension one ($d=1$) with an initial value that is chosen randomly such that for $\ell=0$, $s=5$ and $\ep=0.01$
\[
\norm{u_K^0}_0 = \rho \myand \norm{\Fc_{\neg\ell}(u_K^0)}_s = \ep,
\]
i.e., the initial value is, in the $H^5$ norm, up to $0.01$ close to the constant plane wave~$\rh$. 

For the numerical discretization with the split-step Fourier method we use $2^5$ points for the Fourier collocation in space ($K=2^4$), and we consider three different step-sizes for the discretization in time: the step-size $h=0.04$ that does not fulfill the step-size restriction~\eqref{eq-strongcfl} of Theorem~\ref{thm-main2} and the slightly larger step-sizes $h=0.042$ and $h=0.044$. 

We have checked numerically for $h=0.04$ and $h=0.044$ that Assumptions~\ref{assum-linearstability} and Assumption~\ref{assum-nonres} of Theorem~\ref{thm-main} are fulfilled for $N\le 5$ with $c_1=0.2$, $\widehat{\ep}=0$, $c_2=8$, $\de_2=0.1$ and $s_2=5N$ for $h=0.04$ and $s_2=8N/5$ for $h=0.044$. Note, however, that the step-size restriction~\eqref{eq-strongcfl} of Theorem~\ref{thm-main2} is not fulfilled. For Figure~\ref{fig-exp1} we compute the numerical solution with step-size $h=0.04$ on a long time interval $t\le 10^6$ and plot the absolute values of the Fourier coefficients on two subintervals of length $200$. The same is done in Figure~\ref{fig-exp2} with the step-size $h=0.044$. As stated in Theorem~\ref{thm-main} we observe in both cases that the solution stays concentrated in the $\ell$th Fourier mode over long times.

\begin{figure}[t]
\centering
\includegraphics{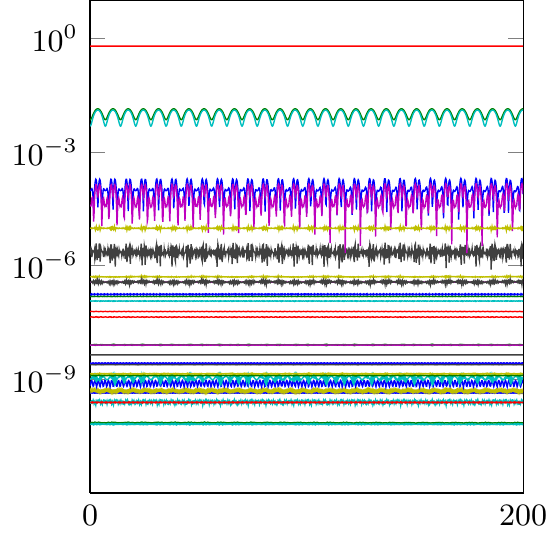} \includegraphics{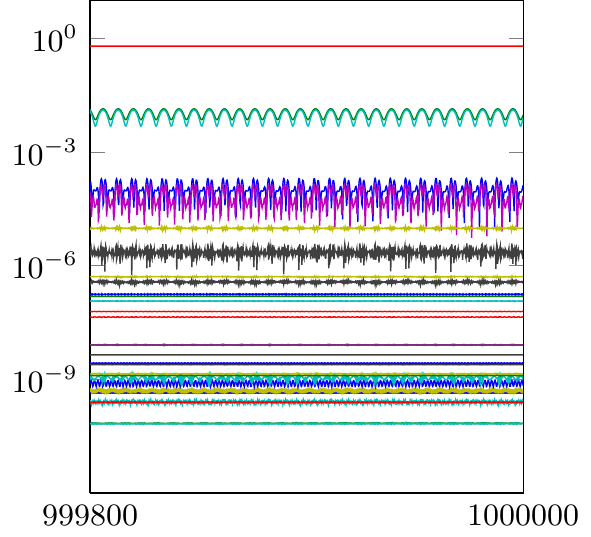}
\caption{Evolution of the absolute values of the Fourier coefficients for $h=0.04$.}\label{fig-exp1}
\end{figure}

\begin{figure}
\centering
\includegraphics{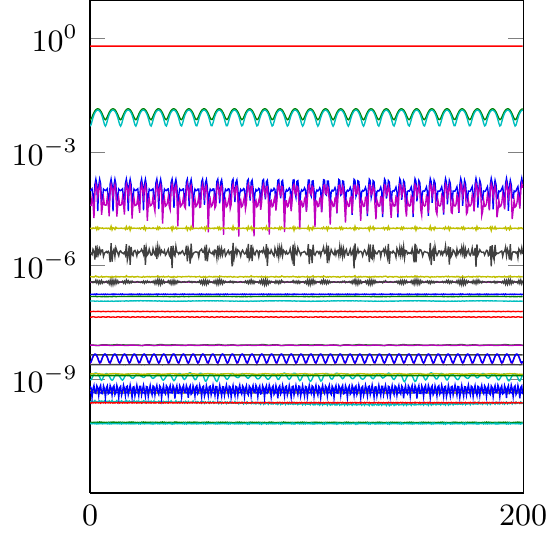} \includegraphics{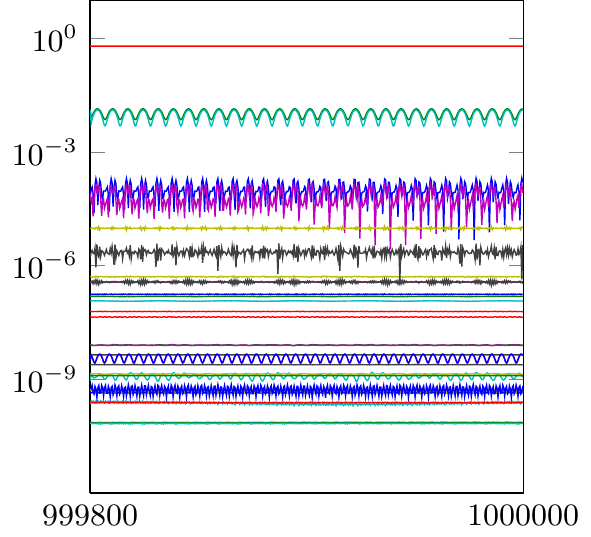}
\caption{Evolution of the absolute values of the Fourier coefficients for $h=0.044$.}\label{fig-exp2}
\end{figure}

For the intermediate step size $h=0.042$, however, Assumption~\ref{assum-linearstability} of Theorem~\ref{thm-main} is not fulfilled. In Figure~\ref{fig-exp3} we again plot the absolute values of the Fourier coefficients of the numerical solution and clearly observe an instability. 

\begin{figure}
\centering
\includegraphics{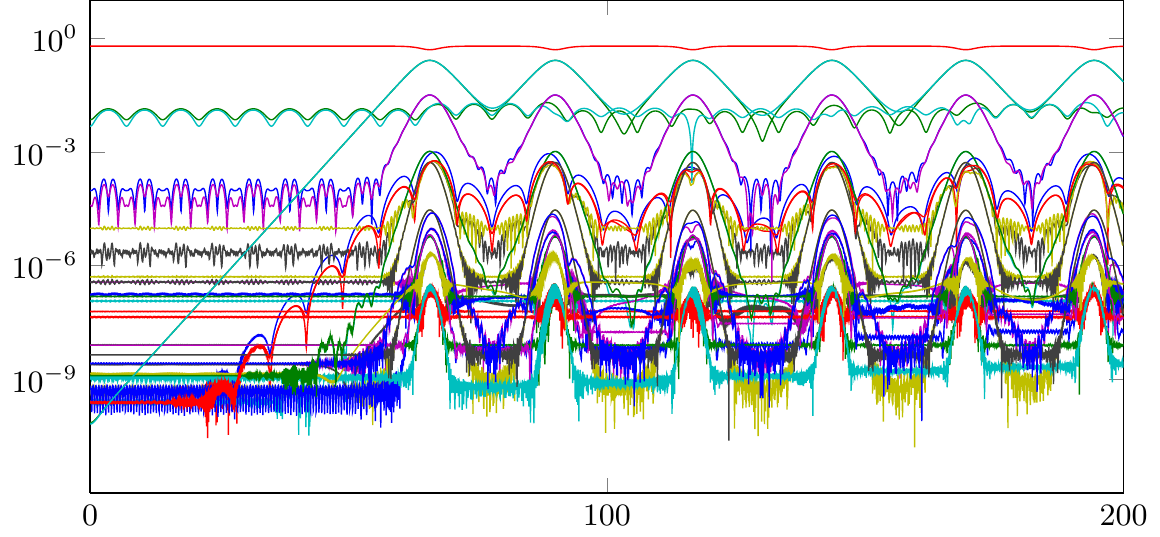}
\caption{Evolution of the absolute values of the Fourier coefficients for $h=0.042$.}\label{fig-exp3}
\end{figure}

\section{Reductions and transformations}\label{sec-trf}

From now on we omit the index $K$ of the numerical solution $u_K^n$, $n=0,1,2,\ldots$. Instead, we denote by $u_j^n$ the $j$th Fourier coefficient of $u^n$: $u^n (x)  = \sum_{j\in\disc} u_j^n \e^{\iu (j\cdot x)}$. We work with the numerical scheme~\eqref{eq-splitstep} in terms of these Fourier coefficients, which takes the form (see \eqref{eq-flow-linear} and \eqref{eq-flow-nonlinear})
\begin{equation}\label{eq-split-u}
u^{n+1}_j = \e^{-\iu \abs{j}^2 h} \sum_{m=0}^{\infty} \frac{(-\iu h\la)^m}{m!} \!\!\!\!\!\!\!\! \sum_{\substack{k^1+\dots+k^{m+1}\\ -l^{1}-\dots-l^{m} \equiv j \bmod{2K}}} \!\!\!\!\!\!\!\! u^n_{k^1} \dotsm u^n_{k^{m+1}} \overline{u}^n_{l^{1}} \dotsm \overline{u}^n_{l^{m}}.
\end{equation}

The goal of this section is to eliminate the $\ell$th Fourier mode, which is not small, from $u^n$. To this end we apply similar reductions and transformations to those for the exact solution in \cite[Sect.~2]{Faoua}, which can be summarised as follows.
\begin{itemize}
\item Transformation $u\leftrightarrow v$ with $u=(u_j)_{j\in\disc}$, $v=(v_j)_{j\in\disc}$: shift to the case $\ell=0$, see Subsect.~\ref{subsec-var-v}.
\item Transformation $v\leftrightarrow (a,\theta,w)$ with $a,\theta\in\R$ and $w=(w_j)_{j\in\disc\setminus\{0\}}$: introduction of polar coordinates $(a,\theta)$ for $v_0$ and rotations $w_j$ of $v_j$, see Subsect.~\ref{subsec-var-wz}.
\item Reduction $(a,\theta,w) \leftrightarrow w$: elimination of $a$ and $\theta$ using conservation of mass and gauge invariance, see Subsect.~\ref{subsec-var-wz}.
\item Transformation $w \leftrightarrow \xi$ with $\xi=(\xi_j)_{j\in\disc\setminus\{0\}}$: diagonalization of the linear part, see Subsects.~\ref{subsec-numfrequ}--\ref{subsec-var-xi}.
\end{itemize}
These transformations and reductions are applied directly to the numerical scheme in the form~\eqref{eq-split-u}. In Subsect.~\ref{subsec-splitstruc-exact}, we consider them from a different perspective, namely from the perspective of the differential equations that form the two steps of the splitting integrator~\eqref{eq-splitstep}. Both perspectives will be important in the following Sect.~\ref{sec-mfe}.

\subsection{Shift to the case \texorpdfstring{$\ell=0$}{l = 0}}\label{subsec-var-v}

We introduce new variables
\[
v_j = u_{\ell+j \bmod{2K}} \myfor j\in\disc.
\]
The numerical scheme \eqref{eq-split-u} in the new variables $v^n$ becomes
\begin{equation}\label{eq-split-v}
v^{n+1}_j = \e^{-\iu \abs{\ell+j\bmod{2K}}^2 h} \sum_{m=0}^{\infty} \frac{(-\iu h\la)^m}{m!} \!\!\!\!\!\!\!\! \sum_{\substack{k^1+\dots+k^{m+1}\\ -l^{1}-\dots-l^{m} \equiv j \bmod{2K}}} \!\!\!\!\!\!\!\! v^n_{k^1} \dotsm v^n_{k^{m+1}} \overline{v}^n_{l^{1}} \dotsm \overline{v}^n_{l^{m}}.
\end{equation}

\subsection{Elimination of the zero mode}\label{subsec-var-wz}

We introduce polar coordinates $(a,\theta)$ for $v_0$,
\[
v_0 = a \e^{\iu\theta} \with a = \abs{v_0},
\]
and new variables $w_j$, $0\ne j\in\disc$, by
\begin{equation}\label{eq-ind}
v_j = w_j\e^{\iu \theta} \myfor j\in\ind = \disc \setminus \{0\}.
\end{equation}
In these new variables $(a,\theta,w)$ with $w=(w_j)_{j\in\ind}$ the numerical scheme \eqref{eq-split-v} becomes
\begin{equation}\label{eq-split-w}\begin{split}
w^{n+1}_j &= \e^{-\iu \abs{\ell+j\bmod{2K}}^2 h} \e^{\iu(\theta^n-\theta^{n+1})} \sum_{m=0}^{\infty} \frac{(-\iu h\la)^m}{m!}\\
 &\qquad \sum_{\substack{k^1+\dots+k^{m+1}\\ -l^{1}-\dots-l^{m} \equiv j \bmod{2K}}} \!\!\!\!\!\!\!\! 
 w^n_{k^1} \dotsm w^n_{k^{m+1}} \overline{w}^n_{l^{1}} \dotsm \overline{w}^n_{l^{m}},
\end{split}\end{equation}
where we use the convention $w_0^n=\overline{w}_0^n=a^n$.

Now, we eliminate $a$ and $\theta$ from \eqref{eq-split-w}. For the elimination of $a$ we observe that the split-step Fourier method \eqref{eq-split-u} conserves mass,
\[
\normbig{ u^{n+1}}_0 = \sum_{j\in\disc} \abs{u_j^{n+1}}^2 = \sum_{j\in\disc} \abs{u_j^{n}}^2 = \normbig{ u^{n}}_0,
\]
a fact that can be easily derived from the representations \eqref{eq-flow-linear} and \eqref{eq-flow-nonlinear} of the flows composing the numerical scheme and the discrete Parseval identity $\sum_{j\in\disc} \abs{u_j^{n}}^2 = (2 K)^{-d} \sum_{k\in\disc} \abs{u^n(x_k)}^2$.
The conservation of mass allows us to express $a^{n}$ in terms of $w_j^{n}$, $j\in\ind$, and $\rh = \norm{u^0}_0$,
\begin{equation}\label{eq-w0}
a^{n} = \Bigl( \rho^2- \sum_{j\in\ind} \abs{w_j^{n}}^2 \Bigr)^{1/2}.
\end{equation}
Also the factor $\e^{\iu(\theta^n-\theta^{n+1})}$ in \eqref{eq-split-w} can be expressed in terms of $w_j$ using \eqref{eq-split-w} for $j=0$:
\begin{equation}\label{eq-theta}\begin{split}
\e^{\iu(\theta^n-\theta^{n+1})} &= \frac{\e^{\iu\abs{\ell}^2h}}{w^{n+1}_0} \sum_{m=0}^{\infty} \frac{(\iu h\la)^m}{m!} \!\!\!\!\!\!\!\! \sum_{\substack{k^1+\dots+k^{m+1}\\ -l^{1}-\dots-l^{m} \equiv 0 \bmod{2K}}} \!\!\!\!\!\!\!\! \overline{w}^n_{k^1} \dotsm \overline{w}^n_{k^{m+1}} w^n_{l^{1}} \dotsm w^n_{l^{m}}.
\end{split}\end{equation}
Hence, $a=w_0=\overline{w}_0$ and $\e^{\iu\theta}$ are determined by $w_j$, $j\in\ind$. The numerical scheme \eqref{eq-split-w} is therefore completely described by the reduced set of variables $w = (w_j)_{j\in\ind}$.

Now, we can replace $w_0^n=a^n$ in \eqref{eq-split-w} and \eqref{eq-theta} by \eqref{eq-w0}. Furthermore, we can use \eqref{eq-w0} with $n+1$ instead of $n$ and $\abs{w_j^{n+1}}^2$ replaced by \eqref{eq-split-w} to replace $w_0^{n+1}$ in \eqref{eq-theta}. For sufficiently small $w$ this leads, after a Taylor expansion of $(\rh^2 - \dots)^{\pm 1/2}$, to an equation for $w^{n+1}$ of the following form (with right-hand side depending only on $w^n$):
\begin{equation}\label{eq-split-z}\begin{split}
w^{n+1}_j &= \e^{-\iu (\abs{\ell+j\bmod{2K}}^2 - \abs{\ell}^2) h} \Bigl( (1-\iu h\la\rh^2)w_j^n - \iu h\la\rh^2 \overline{w}_{-j}^n\\
 &\qquad + \sum_{m+m'=2}^{\infty} \!\!\!\! \sum_{\substack{k^1+\dots+k^m\\ -l^1-\dots-l^{m'} \equiv j \bmod{2K}}} \!\!\!\!\!\!\!\! h \widetilde{Q}_{j,k,l} w^n_{k^1}\dotsm w^n_{k^m} \overline{w}^n_{l^1}\dotsm \overline{w}^n_{l^{m'}} \Bigr).
\end{split}\end{equation}
The  subscripts  here and in the following all belong to the reduced set $\ind$.

\subsection{Linear stability and numerical frequencies}\label{subsec-numfrequ}

The linear part in equation \eqref{eq-split-z} couples $w_j$ to $\overline{w}_{-j}$. This leads us to consider the equation for $w_j$ together with the one for $\overline{w}_{-j}$,
\[
\begin{pmatrix}
w^{n+1}_j\\
\overline{w}^{n+1}_{-j}
\end{pmatrix}
= \e^{-\iu(\abs{\ell+j\bmod{2K}}^2/2 - \abs{\ell-j\bmod{2K}}^2/2)h} A_j \begin{pmatrix}
w^{n}_j\\
\overline{w}^{n}_{-j}
\end{pmatrix}
+ \text{ higher order terms}
\]
with the matrix 
\[
A_{j} = \begin{pmatrix} \alpha_j & \beta_j \\ \overline{\beta_j} &\overline{\alpha_j} \end{pmatrix},
\]
where
\begin{align*}
\al_j &= (1-\iu h\la\rh^2) \e^{-\iu (\abs{\ell+j\bmod{2K}}^2/2 + \abs{\ell-j\bmod{2K}}^2/2 - \abs{\ell}^2)h},\\
\be_j &= -\iu h\la\rh^2 \e^{-\iu (\abs{\ell+j\bmod{2K}}^2/2 + \abs{\ell-j\bmod{2K}}^2/2 - \abs{\ell}^2)h}.
\end{align*}
This matrix has $\abs{\alpha_j}^2-\abs{\beta_j}^2=1$ and its eigenvalues are 
\[
\lambda_j^\pm = \ReT(\alpha_j)\pm \iu \sgn(\ImT(\alpha_j))\sqrt{1-\ReT(\alpha_j)^2}.
\] 
The reason for including $\sgn(\ImT(\alpha_j))$ in the definition of the eigenvalues $\lambda_j^\pm$ will become clear in the following Subsect.~\ref{subsec-var-xi}. 

Assumption~\ref{assum-linearstability} ensures that $\ReT(\alpha_j)^2\le 1$, and hence the eigenvalues $\la_j^\pm$ of $A$ are of modulus one: We have
\[
\e^{-\iu(\abs{\ell+j\bmod{2K}}^2/2 - \abs{\ell-j\bmod{2K}}^2/2)h} \la_j^\pm = \e^{\mp\iu\om_j h}
\]
with the \emph{numerical frequencies} $\om_j$ from \eqref{eq-freq},
\[
\om_j = \sfrac12\abs{\ell+j\bmod{2K}}^2 - \sfrac12\abs{\ell-j\bmod{2K}}^2 + \frac{\arccos(\ReT(\alpha_j))}{-h\sgn(\ImT(\alpha_j))},
\]
where the branch of $\arccos$ with values in $[0,\pi]$ is used. Note that eigenvalues of $A_j$ of modulus greater than one would lead to a growth of the corresponding modes in the linearization of \eqref{eq-split-z}. Assumption~\ref{assum-linearstability} excludes this scenario and thus ensures linear stability of the split-step Fourier method.

\subsection{Diagonalization of the linear part}\label{subsec-var-xi}

We introduce new variables $\xi_j$ that diagonalize the linear part of \eqref{eq-split-z}:
\[
\begin{pmatrix}
\xi_j\\
\overline{\xi}_{-j}
\end{pmatrix} = S_j
\begin{pmatrix}
w_j\\
\overline{w}_{-j}
\end{pmatrix},
\]
where, with the notation of the previous subsection, 
\begin{equation}\label{eq-Sj}
S_j^{-1} = \frac{1}{\sqrt{\abs{\beta_j}^2-\abs{\la_j^+ -\alpha_j}^2}} \begin{pmatrix}
\beta_j & \la_j^- - \overline{\alpha_j}\\
\la_j^+ - \alpha_j & \overline{\beta_j}
\end{pmatrix}
\end{equation}
such that
\[
\e^{-\iu(\abs{\ell+j\bmod{2K}}^2/2 - \abs{\ell-j\bmod{2K}}^2/2)h} S_j A_{j} S_j^{-1} = \begin{pmatrix}
\e^{-\iu\om_jh} & 0\\
0 & \e^{\iu\om_jh}
\end{pmatrix}.
\]
Note that 
\begin{equation}\label{eq-aux-diag}
\abs{\beta_j}^2-\abs{\la_j^+ -\alpha_j}^2 = 2\sqrt{1-\ReT(\alpha_j)^2} \Bigl( \abs{\ImT(\alpha_j)} - \sqrt{1-\ReT(\alpha_j)^2} \Bigr) > 0 ,
\end{equation}
and hence this change of variables, which defines $\xi_j$ and $\overline{\xi}_j$, is well defined because of the structure of $S_j$ (this is the reason for including the sign of $\ImT(\al_j)$ in the definition of the numerical frequencies). Moreover, it is symplectic since $\det(S_j)=1$. With this change of variables, \eqref{eq-split-z} is transformed to
\begin{equation}\label{eq-split-xi}
\xi^{n+1}_j = \e^{-\iu\om_j h} \xi_j^n  + \sum_{m+m'=2}^{\infty} \!\!\!\! \sum_{\substack{k^1+\dots+k^m\\ -l^1-\dots-l^{m'} \equiv j \bmod{2K}}} \!\!\!\!\!\!\!\!\!\!\!\! h Q_{j,k,l} \xi^n_{k^1}\dotsm \xi^n_{k^m} \overline{\xi}^n_{l^1}\dotsm \overline{\xi}^n_{l^{m'}}.
\end{equation}

\subsection{The splitting structure of the numerical scheme in the new variables}\label{subsec-splitstruc-exact}

Recall that in the original variables $u$
\[
u^{n+1} = \Ph_{\text{linear}}^h \circ \Ph_{\text{nonlinear}}^h (u^n),
\]
see equation~\eqref{eq-splitstep}. Here, $\Ph_{\text{linear}}^h=\Ph_{\breve{H}_0}^h$ is the flow at time $h$ of the Hamiltonian differential equation with Hamiltonian function $\breve{H}_0$,
\[
\iu \dot{u}_j = \abs{j}^2 u_j = \frac{\partial \breve{H}_0}{\partial \overline{u}_j} (u,\overline{u}) \with \breve{H}_0(u,\overline{u}) = \sum_{j} \abs{j}^2 u_{j} \overline{u}_{j}.
\]
Correspondingly, $\Ph_{\text{nonlinear}}^h = \Ph_{\breve{P}}^h$ is the flow at time $h$ of the Hamiltonian differential equation with Hamiltonian function $\breve{P}$,
\[
\iu \dot{u}_j = \frac{\partial \breve{P}}{\partial \overline{u}_j} (u,\overline{u}) \with \breve{P}(u,\overline{u}) = \frac{\la}{2} \sum_{j^1+j^2-j^3-j^4 \equiv 0 \bmod{2K}} \!\!\!\! u_{j^1}u_{j^2} \overline{u}_{j^3}\overline{u}_{j^4}.
\]
Now, we consider the transformations $u \leftrightarrow v \leftrightarrow (a,\theta,w) \leftrightarrow w \leftrightarrow \xi$ from the previous subsections on the level of these differential equations (instead of their flows, as we have done in the previous subsections).

\medskip\noindent{\bf Shift $u\leftrightarrow v$.} 
After the change of variables $u\leftrightarrow v$ described in Subsect.~\ref{subsec-var-v} (leading to the numerical scheme~\eqref{eq-split-v}) the splitting scheme becomes
\[
v^{n+1} = \Ph^h_{\check{H}_0} \circ \Ph_{\check{P}}^h (v^n)
\]
with the Hamiltonian functions
\[
\check{H}_0 (v,\overline{v}) = \sum_{j} \abs{\ell+j \bmod{2K}}^2 v_{j} \overline{v}_{j}
\]
and 
\[
\check{P} (v,\overline{v}) = \frac{\la}{2} \sum_{j^1+j^2-j^3-j^4 \equiv 0 \bmod{2K}} \!\!\!\! v_{j^1}v_{j^2} \overline{v}_{j^3}\overline{v}_{j^4}.
\]

\medskip\noindent{\bf Transformation $v\leftrightarrow (a,\theta,w)$.} 
In the variables $(a,\theta,w)$ introduced at the beginning of Subsect.~\ref{subsec-var-wz} (leading to the numerical scheme~\eqref{eq-split-w}), the flow of the Hamiltonian differential equation with Hamiltonian function $\check{H}_0$ has to be replaced by the flow of
\begin{equation}\label{eq-aux2-wa}
\iu \dot{w}_j = \dot{\theta} w_j + \abs{\ell+j\bmod{2K}}^2 w_j.
\end{equation}
The corresponding equation for $a=w_0$ becomes, after taking the real part,
\begin{equation}\label{eq-aux2-wa2}
0 = \dot{\theta} a + \abs{\ell}^2 a.
\end{equation}
Correspondingly, the flow of the Hamiltonian differential equation with Hamiltonian function $\check{P}$ has to be replaced by the flow of
\begin{equation}\label{eq-aux-wa}
\iu \dot{w}_j = \dot{\theta} w_j + \frac{\partial \widehat{P}}{\partial \overline{w}_j}(a,\theta,w,\overline{w}) \with \widehat{P}(a,\theta,w,\overline{w}) = \check{P}(v,\overline{v}).
\end{equation}
Note that the function $\widehat{P}$ is actually independent of $\theta$ (gauge invariance). The equation for $a=w_0$ becomes, after taking the real part,
\begin{equation}\label{eq-aux-wa2}
0 = \dot{\theta} a + \frac12 \frac{\partial \widehat{P}}{\partial a}(a,\theta,w,\overline{w}).
\end{equation} 

\medskip\noindent{\bf Reduction $(a,\theta,w)\leftrightarrow w$.} 
Solving \eqref{eq-aux2-wa2} for $\dot{\theta}$ and inserting this into the equations \eqref{eq-aux2-wa} for $j\ne 0$ shows that \eqref{eq-aux2-wa} becomes, in the reduced set of variables $w$ from Subsect.~\ref{subsec-var-wz} (with the numerical scheme~\eqref{eq-split-z}),
\[
\iu \dot{w}_j =\frac{\partial \widetilde{H}_0}{\partial \overline{w}_j}( w,\overline{w}) \with \widetilde{H}_0(w,\overline{w}) = \sum_{j} \bigl(\abs{\ell+j \bmod{2K}}^2 - \abs{\ell}^2 \bigr) w_{j} \overline{w}_{j}.
\]
Solving \eqref{eq-aux-wa2} for $\dot{\theta}$ and inserting this into the equations \eqref{eq-aux-wa} for $j\ne 0$ yields
\[
\iu \dot{w}_j = \frac{\partial \widehat{P}}{\partial \overline{w}_j}(a,\theta, w,\overline{w}) + \frac{-w_j}{2a} \frac{\partial \widehat{P}}{\partial a}(a,\theta, w,\overline{w}).
\]
Using 
\[
\frac{\partial a}{\partial \overline{w}_j}(w,\overline{w}) = \frac{-w_j}{2 a}
\]
with $a$ given by \eqref{eq-w0}, we see that the equation~\eqref{eq-aux-wa} becomes, in the reduced set of variables $w$ from Subsect.~\ref{subsec-var-wz} (with the numerical scheme~\eqref{eq-split-z}),
\[
\iu \dot{w}_j =\frac{\partial \widetilde{P}}{\partial \overline{w}_j}( w,\overline{w}) \with \widetilde{P}(w,\overline{w}) = \widehat{P}(a,\theta,w,\overline{w}),
\]
which surprisingly is again of Hamiltonian form. We hence have
\[
w^{n+1} = \Ph^h_{\widetilde{H}_0} \circ \Ph_{\widetilde{P}}^h (w^n).
\]
The splitting integrator in the reduced set of variables $w$ is still a Hamiltonian splitting, a splitting into two Hamiltonian equations.

\medskip\noindent{\bf Transformation $w\leftrightarrow \xi$.} 
Concerning the final change of variables $w\leftrightarrow\xi$ of Subsect.~\ref{subsec-var-xi} (leading to the numerical scheme~\eqref{eq-split-xi}) we note first that the matrix $S_j$ was chosen in such a way that it is symplectic. We therefore end up with
\begin{subequations}\label{eq-split-diffeq}
\begin{equation}
\xi^{n+1} = \Ph_{H_0}^h \circ \Ph_{P}^h (\xi^n)
\end{equation}
with
\begin{equation}
H_0(\xi,\overline{\xi}) = \widetilde{H}_0(w,\overline{w}) \myand P(\xi,\overline{\xi}) = \widetilde{P}(w,\overline{w}).
\end{equation}
\end{subequations}

While it is an obvious observation that the numerical scheme in the new variables $\xi$ is still a splitting scheme, it is highly remarkable that the split equations retain their Hamiltonian structure. 

By virtue of the expansion \eqref{eq-split-xi}, we have a concrete expression for the flow $\Ph_{P}^h(\xi^n) = \bigl(\Ph_{H_0}^h\bigr)^{-1} (\xi^{n+1})$,
\begin{equation}\label{eq-splittingstructure}
\Ph_{P}^h (\xi^n)
 = S_j
\begin{pmatrix}
\e^{\iu(\abs{\ell+j\bmod{2K}}^2-\abs{\ell}^2) h} & 0\\
0 & \e^{-\iu(\abs{\ell+j\bmod{2K}}^2-\abs{\ell}^2) h}
\end{pmatrix}
S_j^{-1} 
\begin{pmatrix}
\xi_j^{n+1}\\
\overline{\xi}_{-j}^{n+1}
\end{pmatrix}.
\end{equation}
For later purposes we also introduce an expansion
\begin{equation}\label{eq-P}
P(\xi,\overline{\xi}) = \sum_{m+m'=2}^\infty \sum_{k\in\ind^m,\,l\in\ind^{m'}} P_{k,l} \xi_{k^1}\dotsm\xi_{k^m}\overline{\xi}_{l^1}\dotsm\overline{\xi}_{l^{m'}}
\end{equation}
of the Hamiltonian function $P$.

\subsection{Estimates for the transformation and for the transformed equation}

We derive some bounds for the change of variables $u\leftrightarrow \xi$ described in Subsects.~\ref{subsec-var-v}--\ref{subsec-var-xi}. We assume throughout that Assumption~\ref{assum-linearstability} is fulfilled.

Note that $\abs{w_j}=\abs{u_{\ell+j \bmod{2K}}}$ for $j\in\ind$, and hence we first consider the last transformation $w\leftrightarrow\xi$ of Subsect.~\ref{subsec-var-xi} described by the matrices $S_j$ \eqref{eq-Sj}.

\begin{lemma}\label{lemma-Sj}
The absolute values of the entries of the matrices $S_j$ and $S_j^{-1}$ are bounded by $\sqrt{1+\rh^2/(2\sqrt{c_1})}$, independently of $j$ and $h$.
\end{lemma}
\begin{proof}
With the notations of Subsect.~\ref{subsec-numfrequ} we have by~\eqref{eq-aux-diag}
\begin{gather*}
\absbigg{\frac{\be_j}{\sqrt{\abs{\be_j}^2-\abs{\la_j^+-\al_j}^2}}}^2 = \frac{\abs{\ImT(\al_j)}+\sqrt{1-\ReT(\al_j)^2}}{2 \sqrt{1-\ReT(\al_j)^2}},\\
\absbigg{\frac{\la_j^+-\al_j}{\sqrt{\abs{\be_j}^2-\abs{\la_j^+-\al_j}^2}}}^2 = \frac{\abs{\ImT(\al_j)}-\sqrt{1-\ReT(\al_j)^2}}{2 \sqrt{1-\ReT(\al_j)^2}}.
\end{gather*}
This proves the statement of the lemma since $\abs{\ImT(\al_j)}\le \sqrt{1-\ReT(\al_j)^2} + h\rh^2$ and since $\sqrt{1-\ReT(\al_j)^2} \ge \sqrt{c_1} h$ by Assumption~\ref{assum-linearstability}.
\end{proof}

Now we consider the norm
\[
\norm{\xi}_s=\Bigl(\sum_{j\in\ind} \abs{j}^{2s} \abs{\xi_j}^2\Bigr)^{1/2},
\]
i.e., $\norm{\xi}_s$ is the Sobolev $H^s$ norm of   the function  $\sum_{j\in\ind} \xi_j \e^{\iu(j\cdot x)}$ as introduced in Subsect.~\ref{subsec-mainresult}. The previous Lemma~\ref{lemma-Sj} implies the following result.

\begin{lemma}\label{lemma-norm-equivalent}
For the change of variables $u\leftrightarrow \xi$ there exist positive constants $\widehat{c}$ and $\widehat{C}$ depending only on $c_1$ and an upper bound of $\rh$ such that
\[
\widehat{c} \norm{\xi}_s \le \norm{\Fc_{\neg \ell}(u)}_s \le \widehat{C} \norm{\xi}_s. \tag*{\qed}
\]
\end{lemma}

In particular, the previous lemma shows that the condition $\norm{\Fc_{\neg \ell}(u^0)}_s \le \ep$ of Theorem~\ref{thm-main} becomes in the new variables $\xi$
\begin{equation}\label{eq-smallxi0}
\norm{\xi^0}_s \le \widehat{c}^{-1} \ep.
\end{equation}

We finally collect some estimates for the nonlinearity in the numerical scheme written in the new variables $\xi$  as given by  \eqref{eq-split-xi}. 

\begin{lemma}\label{lemma-nonlinearity}
The nonlinearity given by the coefficients $Q_{j,k,l}$ in \eqref{eq-split-xi} satisfies for $s>d/2$ 
\begin{align*}
\Bigl( \sum_{j\in\ind} \abs{j}^{2s} \Bigl( \!\!\!\! \sum_{\substack{k^1+\dots+k^m\\ -l^1-\dots-l^{m'} \equiv j \bmod{2K}}} \!\!\!\! \absbig{Q_{j,k,l} \et^1_{k^1}\dotsm \et^m_{k^m} &\et^{m+1}_{l^1}\dotsm \et^{m+m'}_{l^{m'}} } \Bigr)^2 \Bigr)^{1/2}\\
&\le C_{m,m',s} \norm{\et^1}_s \dotsm\norm{\et^{m+m'}}_s
\end{align*}
for vectors $\et^1,\dots,\et^{m+m'}\in \C^{\ind}$. The constants $C_{m,m',s}$ depend only on $m$, $m'$, $s$, $c_1$ and $\rh$ and satisfy
\[
\sum_{m+m'=2}^{\infty} C_{m,m',s} r^{m+m'} \le C
\]
for some positive constants $r$ and $C$ depending only on $c_1$, $s$ and $\rho$.
\end{lemma}
\begin{proof}
(a) By carefully going through the construction of the coefficients $Q_{j,k,l}$ in Subsects.~\ref{subsec-var-wz} and~\ref{subsec-var-xi} one shows for the coefficients $Q_{j,k,l}$ that there exists a constant $C$ depending on $c_1$ and $\rh$ such that
\begin{equation}\label{eq-lemmanonlinearity-aux}
\abs{Q_{j,k,l}} \le C^{m+m'}
\end{equation}
for all $j\in\ind$, all $k\in\ind^m$ and all $l\in\ind^{m'}$. 

(b) The first estimate of the lemma follows by applying~\eqref{eq-lemmanonlinearity-aux} and the Cauchy-Schwarz inequality and by using
\[
\sum_{\substack{k^1+\dots+k^m\\ -l^1-\dots-l^{m'} \equiv j \bmod{2K}}} \biggl( \frac{\abs{j}}{\abs{k^1}\dotsm\abs{k^m}\abs{l^1}\dotsm\abs{l^{m'}}} \biggr)^{2s} \le c^{m+m'} 
\]
with a constant $c$ depending only on $s>d/2$. The second estimate of the lemma then follows also from~\eqref{eq-lemmanonlinearity-aux}.
\end{proof}

\section{Modulated Fourier expansions}\label{sec-mfe}

In this section we will prove Theorem~\ref{thm-main} using \emph{modulated Fourier expansions} originally introduced in \cite{Hairer2000}, see also \cite{Hairer2006}. Throughout we will work with the numerical scheme in the new variables $\xi$   introduced in Sect.~\ref{sec-trf}, see \eqref{eq-split-xi}.

There are two main steps:
\begin{itemize}
\item Construction of a short-time approximation of $\xi^n$ from \eqref{eq-split-xi} by a modulated Fourier expansion in Subsects.~\ref{subsec-mfe}--\ref{subsec-defect}. 
\item Almost-invariants of the modulated Fourier expansion that allow us to prove a result on a long time interval in Subsects.~\ref{subsec-splitstruc-mfe}--\ref{subsec-longtimenearcons}. 
\end{itemize}
For the first main step it is convenient to work with the numerical scheme as given by the composition of flows \eqref{eq-split-xi}, whereas for the derivation of the almost-invariants it is necessary to switch to the level of the differential equations whose flows compose the numerical scheme~\eqref{eq-split-diffeq}. We ultimately show that $\norm{\xi^n}_s$ stays of order $\ep$ for initial values $\xi^0$ of order $\ep$.
This preservation of smallness and regularity of $\xi^n$ is the main ingredient for the final proof of Theorem~\ref{thm-main} in Subsect.~\ref{subsec-proof1}. 

The proof via modulated Fourier expansions given here uses and combines ideas from several previous proofs using such expansions: The aforementioned  idea of switching between the flows and the differential equations is loosely based on \cite{GaucklerDiss,Gauckler2010b}, the construction of the modulated Fourier expansion with an asymptotic expansion is based on \cite{Gauckler2012,Hairer2012}, the idea of using modified frequencies $\omt_j$ instead of the original (numerical) frequencies $\om_j$ of \eqref{eq-freq}  for the modulated Fourier expansion is also used in \cite{Gaucklera}, the non-resonance condition in Assumption~\ref{assum-nonres} is used in a similar way to \cite{Cohen2008}, and the use of almost-invariants of the modulated Fourier expansion to prove long-time almost-conservation properties can be traced back to \cite{Hairer2000}. 

In the following analysis, the (generic) constants $C$, $s_0$ and $\de_0$ are all independent of the small parameters $\ep$ from \eqref{eq-smallxi0} and $\widehat{\ep}$ from Assumption~\ref{assum-nonres}. The constants $C$ and $\de_0$ will depend on the constants $c_1$, $c_2$, $s_2$, $\de_2$ and $N$ of Assumptions~\ref{assum-linearstability} and \ref{assum-nonres}, on $s$ from \eqref{eq-smallxi0}, on an upper bound of $\rh=\norm{u^0}_0$, on the index $\ell\in\disc$ from~\eqref{eq-pwnum} and on the dimension $d$. The constant $s_0$ will depend only on $d$ and $s_2$.

\subsection{Resonant modulated Fourier expansion}\label{subsec-mfe}

In order to motivate the modulated Fourier expansion we consider here, let us first have a look at \eqref{eq-split-xi} in the linear case (all $Q_{j,k,l}=0$). In this case, the evolution of the $j$th mode is given by the multiplication with $\e^{-\iu\om_j t}$. In the presence of the nonlinearity, we seek for an expansion, the modulated Fourier expansion, in terms of products of these exponentials that are multiplied (modulated) by slowly varying coefficients. 

There are two pitfalls in the present situation that have to be handled with care. First, it turns out that the frequencies $\om_j$ of \eqref{eq-freq} are inconvenient when it comes to resonance issues. Therefore we use the modified frequencies $\omt_j$ of Assumption~\ref{assum-nonres} instead and consider products of the exponentials $\e^{-\iu\omt_j t}$:
\[
\e^{- \iu (\kbf\cdot\omtbf) t} \with \kbf\cdot\omtbf = \sum_{j\in\ind} k_j\omt_j
\]
for vectors of integers $\kbf= (k_j)_{j\in\ind}\in\Z^\ind$ and the vector $\omtbf=(\omt_j)_{j\in\ind}$ of modified frequencies.

Second, the modified frequencies $\omt_j$ of Assumption~\ref{assum-nonres} are by definition exactly resonant, for instance $\omt_j = \omt_l$ for $\abs{j}=\abs{l}$ in the case $\ell=0$. Hence, we cannot distinguish all products $\e^{- \iu (\kbf\cdot\omtbf) t}$, and we therefore introduce the \emph{resonance module}
\[
\res = \bigl\{\, \kbf\in\Z^\ind : \kbf\cdot\omtbf = 0, \, j(\kbf) = 0 \,\bigr\},
\]
where
\[
j(\kbf) = \sum_{l\in\ind}k_l l \bmod{2K}.
\]
The restriction $j(\kbf) = 0$ in the definition of the resonance module comes from the fact that the products $\e^{- \iu (\kbf\cdot\omtbf) t}$ are attached to some specific mode $\xi_j$, namely $j=j(\kbf)$, as we will see in the following.

With these preliminaries, we introduce the \emph{resonant modulated Fourier expansion}
\begin{equation}\label{eq-mfe}
\xi_j(t) = \sum_{[\kbf]} z_j^{[\kbf]}(\de t) \e^{- \iu (\kbf\cdot\omtbf) t}
\end{equation}
Here,
\begin{equation}\label{eq-delta}
\de = \max(\ep , \widehat{\ep})^{1/2}
\end{equation}
is a small parameter and the sum is over all residue classes $[\kbf]\in\Z^\ind / \res$. The coefficients of the modulated Fourier expansion, the \emph{modulation functions} $z_j^{[\kbf]}$, are required to be polynomials on a slow time scale $\ta=\de t$ with $\de$ from \eqref{eq-delta} that have all derivatives bounded independently of the small parameters. By a slight abuse of notation we write in the following $z_j^\kbf$ instead of $z_j^{[\kbf]}$ and $\sum_{\kbf}$ instead of $\sum_{[\kbf]}$.

\subsection{Modulation equations}\label{subsec-modequations}

Requiring $\xi_j(t_n) = \xi_j^n$ for $n\ge 1$ with $\xi_j^n$ given by \eqref{eq-split-xi} yields, after a comparison of the coefficients of $\e^{\iu (\kbf\cdot\omtbf) t}$, \emph{modulation equations} for the modulation functions $z_j^\kbf$:
\begin{subequations}\label{eq-modsystem}
\begin{equation}\label{eq-modsystem-motion}\begin{split}
z_j^\kbf(\ta+\de h) \e^{-\iu (\kbf\cdot\omtbf) h} &= \e^{-\iu \om_j h} z_j^\kbf(\ta) + \sum_{m+m'=2}^{\infty} \sum_{\substack{\kbf^1+\dots+\kbf^m\\ - \lbf^1-\dots-\lbf^{m'}\in[\kbf]}}\\
& \sum_{k\in\ind^m,\, l\in\ind^{m'}} h Q_{j,k,l} z_{k^1}^{\kbf^1}(\ta) \dotsm z_{k^m}^{\kbf^m}(\ta) \overline{z}_{l^1}^{\lbf^1}(\ta)\dotsm \overline{z}_{l^{m'}}^{\lbf^{m'}}(\ta).
\end{split}\end{equation}
The condition $\xi_j(0) = \xi_j^0$ yields
\begin{equation}\label{eq-modsystem-init}
\sum_{\kbf} z_j^{\kbf}(0) = \xi_j^0.
\end{equation}
\end{subequations}

For the approximate solution of the modulation equations~\eqref{eq-modsystem} it is useful to expand the modulation functions in powers of $\ep$ and $\de$,
\begin{equation}\label{eq-mfexp}
z_j^\kbf(\ta) = \sum_{p=0}^\infty \ep \de^{p} z_{j,p}^\kbf(\ta)
\end{equation}
with polynomials $z_{j,p}^\kbf$ in $\ta = \de t$. We call the functions $z_{j,p}^\kbf$ \emph{modulation coefficient functions} and set $z_{j,p}^\kbf = 0$ for $p<0$. 
After dividing by $\de h\e^{-\iu(\kbf\cdot\omtbf)h}$, expanding $z_j^\kbf(\ta+\de h)$ around $\ta$ and (formally) comparing the coefficients of $\ep\de^{p}$, the modulation equations \eqref{eq-modsystem-motion} become
\begin{subequations}\label{eq-modsys}
\begin{equation}\label{eq-modsys-eqmot}\begin{split}
&\frac{1 - \e^{-\iu(\om_j-\kbf\cdot\omtbf)h}}{\de h} \, z_{j,p}^\kbf + \dot{z}_{j,p}^\kbf =  - \sum_{r = 2}^{\infty} \frac{ h^{r-1}}{r!} \frac{\dd^r}{\dd\ta^r} z_{j,p+1-r}^\kbf\\
 &\qquad\qquad   + \sum_{m+m'=2}^{\infty} \sum_{\substack{p_1+\dots+p_m\\ + q_1+\dots+q_{m'}=p+3-2(m+m')}} \frac{\ep^{m+m'-1}}{\de^{2m+2m'-2}} \sum_{\substack{\kbf^1+\dots+\kbf^m\\ - \lbf^1-\dots-\lbf^{m'}\in[\kbf]}} \\
&\qquad\qquad\qquad\qquad  \sum_{k\in\ind^m,\, l\in\ind^{m'}} \e^{\iu (\kbf\cdot\omtbf) h} Q_{j,k,l} z_{k^1,p_1}^{\kbf^1} \dotsm z_{k^m,p_m}^{\kbf^m} \overline{z}_{l^1,q_1}^{\lbf^1}\dotsm \overline{z}_{l^{m'},q_{m'}}^{\lbf^{m'}}.
\end{split}\end{equation}
Condition \eqref{eq-modsystem-init} yields
\begin{equation}\label{eq-modsys-init}
z_{j,p}^\jvec(0) = - \sum_{\kbf\ne \jvec} z_{j,p}^{\kbf}(0) + \begin{cases}
\ep^{-1}\xi_j^0, & p = 0,\\
0, & p>0,
\end{cases}
\end{equation}
where $\jvec$ denotes the $j$th unit vector in $\Z^\ind$.
\end{subequations}

\subsection{Construction of modulation functions}\label{subsec-constr}

We construct modulation functions $z_j^\kbf$ that solve the modulation equations \eqref{eq-modsystem} up to a small defect. We work with the asymptotic expansion~\eqref{eq-mfexp} and consider the equations~\eqref{eq-modsys}. The crucial observation is that the right-hand side of \eqref{eq-modsys-eqmot} depends only on modulation coefficient functions $z_{k,q}^\lbf$ with $q<p$. This allows us to solve the equations \eqref{eq-modsys} up to a small defect by the following simple recursion. 

Fix $p\ge 0$ and assume that we have computed all modulation coefficient functions $z_{j,q}^\kbf$ with $q<p$ (this is true for $p=0$). Equation \eqref{eq-modsys-eqmot} is then of the form
\[
\al z_{j,p}^\kbf + \dot{z}_{j,p}^\kbf = P
\]
with a polynomial $P$. The unique polynomial solution of this equation is given for $\al\ne 0$ by
\begin{equation}\label{eq-polsol}
z_{j,p}^\kbf(\ta) = \sum_{m=0}^{\deg(P)} (-1)^m \al^{-m-1} \frac{\dd^m}{\dd\ta^m}P(\ta).
\end{equation}
We therefore compute $z_{j,p}^\kbf$ for all $j$ and all $\kbf$ as follows. 
\begin{subequations}\label{eq-construction}
\begin{enumerate}
\item[(i)] For indices $(j,\kbf)$ with $j\ne j(\kbf)$ or $\norm{\widetilde{\kbf}}> p$ for all $\widetilde{\kbf}\in[\kbf]$ we set
\begin{equation}\label{eq-construction-singlewavelargek}
z_{j,p}^\kbf = 0.
\end{equation}
This is consistent with \eqref{eq-modsys-eqmot} since the right-hand side of this equation vanishes for these indices by induction (recall that $Q_{j,k,l}=0$ if $j\not\equiv k^1+\dots+k^m-l^1-\dots-l^{m'} \bmod{2K}$).
\item[(ii)] For indices $(j,\kbf)$ with $\abs{1 - \e^{-\iu(\om_j-\kbf\cdot\omtbf)h}}\ge \de h/2$ that are not covered by (i) we 
\begin{equation}\label{eq-construction-offdiag}
\text{compute $z_{j,p}^\kbf$ from \eqref{eq-modsys-eqmot} and \eqref{eq-polsol}.}
\end{equation}
Indeed, the factor in front of $z_{j,p}^\kbf$ in \eqref{eq-modsys-eqmot} is bounded for these indices away from zero, and the comparison of coefficients used to derive \eqref{eq-modsys-eqmot} thus makes sense.
\item[(iii)] For indices $(j,\kbf)\ne(j,\jvec)$ that are neither covered by (i) nor by (ii) we set
\begin{equation}\label{eq-nearres}
z_{j,p}^\kbf = 0.
\end{equation}
Of course, this introduces a defect which, however, can be controlled using the non-resonance condition of Assumption~\ref{assum-nonres} as we shall see in Subsect.~\ref{subsec-defect}.\\ For the considered indices $(j,\kbf)$ we have $\abs{1 - \e^{-\iu(\om_j-\kbf\cdot\omtbf)h}}< \de h/2$, and they are in this sense close to a resonance. We therefore call them \emph{near-resonant} in the following.
\item[(iv)] Having computed $z_{j,p}^\kbf$ for all $j$ and all $\kbf\ne\jvec$ in (i)--(iii) we can
\begin{equation}\label{eq-construction-diag-init}
\text{compute $z_{j,p}^\jvec(0)$ from \eqref{eq-modsys-init}.}
\end{equation}
Moreover, since the factor in front of $z_{j,p}^\kbf$ in \eqref{eq-modsys-eqmot} vanishes for $\kbf=\jvec$, we can
\begin{equation}\label{eq-construction-diag-der}
\text{compute $\dot{z}_{j,p}^\jvec$ from \eqref{eq-modsys-eqmot}.}
\end{equation}
This allows us to compute the diagonal modulation coefficient functions $z_{j,p}^\jvec$. 
\end{enumerate}
We stop the above construction \eqref{eq-construction} of modulation coefficient function $z_{j,p}^\kbf$ after $p=N$,
\begin{equation}\label{eq-cutoff}
z_{j,p}^\kbf = 0 \myfor p> N.
\end{equation}
\end{subequations}

It is clear that the construction leads to modulation coefficient functions $z_{j,p}^\kbf$ that are polynomials in $\ta$, of degree bounded by $p$. Moreover, we have
\begin{equation}\label{eq-p0}
z_{j,0}^\kbf = 0 \myfor \kbf\ne\jvec
\end{equation}
because the right-hand side of \eqref{eq-modsys-eqmot} vanishes for $p=0$.

\subsection{Size of the modulation functions}\label{subsec-size}

We estimate the modulation coefficient functions constructed in \eqref{eq-construction}. For fixed index $p$ we collect them in the vectors
\[
\zbf_p=(z^{\kbf}_{j,p})_{j\in\ind,\kbf\in\Z^\ind}.
\]
We also consider their rescalings
\begin{equation}\label{eq-Ga}
(\Gabf^{s-\widehat{s}}\zbf_p)^{\kbf}_{j} := \bigl(\Ga^\kbf\bigr)^{s-\widehat{s}} \cdot z^{\kbf}_{j,p} \with \Ga^\kbf := \min_{\widetilde{\kbf}\in[\kbf]} \Bigl( 2^{\norm{\widetilde{\kbf}}} \prod_{l\in\ind} \abs{l}^{\abs{\widetilde{k}_l}} \Bigr)
\end{equation}
and with $s\ge \widehat{s} := (d+1)/2$ such that Lemma~\ref{lemma-nonlinearity} is applicable for $s$ and $\widehat{s}$.

For vectors $\vbf=(v^\kbf_j)_{j\in\ind,\kbf\in\Z^\ind}$ of polynomials $v^{\kbf}_j=v^\kbf_j(\ta)$ in $\ta$ we use the norm 
\[
\normv{\vbf}_{s,\ta} = \normBig{\Bigl( \sum_{\kbf} \abs{v^{\kbf}_{j}}_{\ta} \Bigr)_{j\in\ind} }_{s} = \biggl( \sum_{j\in\ind} \abs{j}^{2s} \Bigl( \sum_{\kbf} \abs{v^{\kbf}_{j}}_{\ta} \Bigr)^2 \biggr)^{1/2},
\]
where
\[
\abs{v}_{\ta} = \sum_{m=0}^{\infty} \frac{1}{m!} \absbigg{\frac{\dd^m}{\dd\ta^m}v(\ta)}.
\]

\begin{lemma}\label{lemma-size}
The modulation coefficient functions \eqref{eq-construction} satisfy on $0\le\ta\le 1$ for $\de\le\de_0$ and $s\ge\widehat{s}$
\[
\normv{ \zbf_p }_{s,\ta} \le C \myand \normv{\Gabf^{s-\widehat{s}}\zbf_p}_{\widehat{s},\ta} \le C
\]
for all $p$ with constants $C$ and $\de_0$.
\end{lemma}
\begin{proof}
This follows from the recursive construction \eqref{eq-construction}: The property
\[
\abs{vw}_{\ta} \le \abs{v}_{\ta} \abs{w}_{\ta}
\]
together with Lemma~\ref{lemma-nonlinearity} yields inductively an estimate of the nonlinearity on the right-hand side of \eqref{eq-modsys-eqmot} in the norm $\normv{\cdot}_{s,\ta}$ (note that terms in the nonlinearity with $2(m+m')>p+3$ vanish, and hence the sum over $m$ and $m'$ is finite). Then, the property $\abs{\dot{v}}_{\ta} \le \deg(v) \abs{v}_{\ta}$ allows us to estimate the norm $\normv{\cdot}_{s,\ta}$ for the vector consisting of modulation coefficient functions constructed with \eqref{eq-construction-offdiag}. For the remaining nonzero modulation coefficient functions constructed with \eqref{eq-construction-diag-init}--\eqref{eq-construction-diag-der}, the estimate in the norm $\normv{\cdot}_{s,\ta}$ then follows using the smallness of the initial value \eqref{eq-smallxi0} and the property $\abs{v}_{\ta}\le \abs{v(0)} + \sup_{0\le\widetilde{\ta}\le\ta} \abs{\dot{v}}_{\widetilde{\ta}}$. 

For the estimate of the rescaling $\Gabf^{s-\widehat{s}}\zbf_p$ in the norm $\normv{\cdot}_{\widehat{s},\ta}$ we can use essentially the same argument: We just have to take into account that 
\begin{equation}\label{eq-Ga-prod}
\Ga^{\kbf^1+\kbf^2} \le \Ga^{\kbf^1} \Ga^{\kbf^2}
\end{equation}
and that $\Ga^{\jvec} = 2\abs{j}$. The latter follows from
\[
\abs{j} = \abs{j(\widetilde{\kbf})} \le \sum_{l\in\ind} \abs{\widetilde{k}_l} \abs{l} \le \norm{\widetilde{\kbf}} \prod_{l\in\ind}\abs{l}^{\abs{\widetilde{k}_l}}
\]
for $\widetilde{\kbf}\in[\jvec]$ and $j\in\ind$. 
\end{proof}

\subsection{Defect and error}\label{subsec-defect}

The modulation functions constructed in \eqref{eq-construction} via their modulation coefficient functions \eqref{eq-mfexp} are supposed to fulfill the modulation system~\eqref{eq-modsystem}. However, there are two sources of error in their construction: First, we stopped the construction of modulation coefficient functions $z^{\kbf}_{j,p}$ after $p=N$ \eqref{eq-cutoff}. Second, the modulation functions for near-resonant indices $(j,\kbf)$ were set to zero \eqref{eq-nearres}. In other words, the constructed modulation functions satisfy the equations of motion~\eqref{eq-modsystem-motion} of the modulation system only up to a defect,
\begin{equation}\label{eq-defect}\begin{split}
d^{\kbf}_j + e^{\kbf}_j &= - z_j^\kbf(\cdot +\de h) \e^{-\iu (\kbf\cdot\omtbf) h} + \e^{-\iu \om_j h} z_j^\kbf + \sum_{m+m'=2}^{\infty} \sum_{\substack{\kbf^1+\dots+\kbf^m\\ - \lbf^1-\dots-\lbf^{m'}=[\kbf]}}\\
&\qquad\qquad\qquad\qquad\qquad\qquad \sum_{k\in\ind^m,\, l\in\ind^{m'}} h Q_{j,k,l} z_{k^1}^{\kbf^1} \dotsm z_{k^m}^{\kbf^m} \overline{z}_{l^1}^{\lbf^1} \dotsm \overline{z}_{l^{m'}}^{\lbf^{m'}},
\end{split}\end{equation}
whereas the initial condition~\eqref{eq-modsystem-init} is met exactly.
Here, $\dbf = (d^{\kbf}_j)_{j\in\ind,\kbf\in\Z^{\ind}}$ denotes the defect from the cut-off~\eqref{eq-cutoff}, i.e.,
\begin{equation}\label{eq-defect-cutoff}
d^{\kbf}_j = \sum_{p=N+1}^{\infty} \ep \de^{p} h F^{\kbf}_{j,p} \e^{-\iu(\kbf\cdot\omtbf)h},
\end{equation}
where $F^{\kbf}_{j,p}$ is the right-hand side of \eqref{eq-modsys-eqmot}.
The defect in near-resonant indices that is not yet covered by $\dbf$ is denoted by $\ebf = (e^{\kbf}_j)_{j\in\ind,\kbf\in\Z^{\ind}}$, i.e., $e^{\kbf}_j$ is different from zero only for near-resonant indices and in this case
\begin{equation}\label{eq-defect-nearres}
e^{\kbf}_j = \sum_{p=1}^{N} \ep \de^p h F^{\kbf}_{j,p} \e^{-\iu(\kbf\cdot\omtbf)h}.
\end{equation}
(Recall that $F^{\kbf}_{j,p}=0$ for $p=0$.) Both defects are estimated in the following lemma.

\begin{lemma}\label{lemma-defect}
The defects \eqref{eq-defect}--\eqref{eq-defect-nearres} satisfy on $0\le\ta\le 1$ for $\de\le\de_0$ and $s\ge s_0$
\begin{gather*}
\normv{ \dbf }_{s,\ta} \le C \ep \de^{N+1} h, \qquad \normv{\ebf}_{s,\ta} \le C \ep \de^{N+1} h,\\
\normv{ \Gabf^{s-\widehat{s}}\dbf }_{\widehat{s},\ta} \le C \ep \de^{N+1} h, \qquad \normv{\Gabf^{s-\widehat{s}}\ebf}_{\widehat{s},\ta} \le C \ep \de h
\end{gather*}
with constants $C$, $s_0$ and $\de_0$.
\end{lemma}
\begin{proof}
(a) For the bound of $\dbf$, we note that by Lemmas~\ref{lemma-nonlinearity} and~\ref{lemma-size}
\[
\normv{\dbf}_{s,\ta} \le C\ep\de^{N+1}h + \ep h \sum_{m+m'=2}^{\infty} C_{m,m',s} C^{m+m'} \sum_{p=\max(N+1,2(m+m')-3)}^{\infty} \de^p,
\]
where $C/(N+1)$ denotes the constant of Lemma~\ref{lemma-size}. Splitting the sum over $m$ and $m'$ in a part with $m+m'\le N+3$ and another part with $m+m'\ge N+4$ proves the claimed estimate of $\dbf$ using the second part of Lemma~\ref{lemma-nonlinearity} for the sum with $m+m'\ge N+4$ and sufficiently small $\de$. The estimates of $\Gabf^{s-\widehat{s}}\dbf$ and $\Gabf^{s-\widehat{s}}\ebf$ follow similarly.

(b) Concerning the defect $\ebf$ in near-resonant indices, we note that for those indices $(j,\kbf)$
\begin{align*}
\absbig{\e^{-\iu(\omt_j-\kbf\cdot\omtbf)h} - 1} &\le \absbig{\e^{-\iu\omt_jh} - \e^{-\iu \om_j h}} + \absbig{\e^{-\iu(\om_j-\kbf\cdot\omtbf)h} - 1}\\
 &< 2\absbig{\sin((\omt_j-\om_j)h/2)} + \frac{\de h}{2} \le \widehat{\ep} h  + \frac{\de h}{2} \le \de h
\end{align*}
by Assumption~\ref{assum-nonres} and for $2 \de\le 1$. The non-resonance condition of Assumption~\ref{assum-nonres} (used with $\kbf-\langle j \rangle$ in place of $\kbf$) thus implies 
\[
\abs{j}^{s-\widehat{s}} \le c_2^{(s-\widehat{s})/2} \de^{N} \bigl(\Ga^{\kbf} \bigr)^{s-\widehat{s}}
\]
for $s- \widehat{s}\ge 2s_2$. This shows that
\[
\normv{\ebf}_{s,\ta} \le c_2^{(s-\widehat{s})/2} \de^N \normv{\Gabf^{s-\widehat{s}}\ebf}_{\widehat{s},\ta},
\]
and the claimed estimate of $\ebf$ follows from Lemma~\ref{lemma-size}.
\end{proof}

Now, we study the difference $\xi^n - \xi(t_n)$ of the numerical solution $\xi^n$ of \eqref{eq-split-xi} and its modulated Fourier expansion $\xi(t)$ of \eqref{eq-mfe}. In this modulated Fourier expansion $\xi(t)$ of \eqref{eq-mfe} we use the modulation functions constructed in \eqref{eq-construction} at discrete times $t_n = nh$.

\begin{proposition}\label{prop-error}
We have for $\de\le\de_0$ and $s\ge s_0$
\[
\norm{ \xi^n - \xi(t_n) }_s \le C \ep \de^{N} \myfor 0\le t_n = nh \le \de^{-1}
\]
with constants $C$, $s_0$ and $\de_0$.
\end{proposition}
\begin{proof}
(a) From Lemma~\ref{lemma-size} we know for the modulated Fourier expansion the estimate 
\[
\norm{\xi(t_n)}_s \le C\ep \myfor 0\le t_n \le \de^{-1}.
\]

(b) A corresponding estimate holds, for sufficiently small $\ep$, also for the numerical solution:
\[
\norm{\xi^n}_s \le 2\widehat{c}^{-1}\ep \myfor 0\le t_n \le \ep^{-1/2}
\]
with $\widehat{c}$ from \eqref{eq-smallxi0}, since by Lemma~\ref{lemma-nonlinearity} and induction 
\[
\norm{\xi^{n+1}}_s \le \norm{\xi^0}_s + h \sum_{n'=0}^n \sum_{m+m'=2}^{\infty} C_{m,m',s} \norm{\xi^{n'}}_s^{m+m'} \le \widehat{c}^{-1}\ep + \frac{4C}{\widehat{c}^2 r^2}\, nh \ep^2
\]
for $2\ep \le r\widehat{c}$ with $C$ and $r$ from Lemma~\ref{lemma-nonlinearity}. We may assume without loss of generality that the constant $C$ from (a) is larger than $2/\widehat{c}$.

(c) The modulated Fourier expansion satisfies by \eqref{eq-defect} 
\begin{align*}
\xi_j(t_{n+1}) &= \e^{-\iu \om_j h} \xi_j(t_n) - \sum_{\kbf} \bigl(d_j^\kbf(\de t_n) + e_j^\kbf(\de t_n)\bigr)\e^{-\iu (\kbf\cdot\omtbf) t} + \sum_{m+m'=2}^{\infty}\\
&\qquad\qquad \sum_{k\in\ind^m,\,l\in\ind^{m'}} h Q_{j,k,l} \xi_{k^1}(t_n)\dotsm \xi_{k^m}(t_n) \overline{\xi}_{l^1}(t_n)\dotsm \overline{\xi}_{l^{m'}}(t_n).
\end{align*}
Subtracting the numerical solution $\xi^{n+1}$ of \eqref{eq-split-xi} and using (a), (b), Lemma~\ref{lemma-nonlinearity} and Lemma~\ref{lemma-defect} shows that for $0\le t_n \le \de^{-1}$ and for sufficiently small $\ep$
\[
\norm{ \xi^{n+1} - \xi(t_{n+1}) }_s \le \norm{ \xi^n - \xi(t_n) }_s + C \ep \de^{N+1} h + C \ep h \norm{ \xi^n - \xi(t_n) }_s.
\]
The claimed estimate follows inductively.
\end{proof}

\subsection{The splitting structure of the modulated Fourier expansion}\label{subsec-splitstruc-mfe}

In the previous subsections, a modulated Fourier expansion was constructed and analysed based on the representation \eqref{eq-split-xi} of the numerical scheme, i.e., based on flows of differential equations. Recall that we have derived in Subsect.~\ref{subsec-splitstruc-exact} differential equations (Hamiltonian functions) that underly the flows that compose the numerical scheme. In this subsection, we will derive corresponding differential equations for the modulated Fourier expansion.

Motivated by \eqref{eq-split-diffeq}, we denote by $\Ph^h_{\Hbf_0}$ the flow at time $h$ of the Hamiltonian differential equation
\[
\iu \dot z_j^\kbf = \frac{\partial \Hbf_0}{\partial \overline{z}_j^\kbf} (\zbf,\overline{\zbf})
\]
with Hamiltonian function
\[
\Hbf_0(\zbf,\overline{\zbf}) = \sum_{\kbf} \sum_{j\in\ind} \bigl(\abs{\ell+j\bmod{2K}}^2 - \abs{\ell}^2\bigr) \abs{w_j^\kbf}^2, \qquad 
\begin{pmatrix}
w_j^\kbf\\
\overline{w}_{-j}^{-\kbf}
\end{pmatrix} = S_j^{-1}
\begin{pmatrix}
z_j^\kbf\\
\overline{z}_{-j}^{-\kbf}
\end{pmatrix}.
\]
Correspondingly, we denote by $\Ph^h_{\Pbf}$ the flow at time $h$ of the Hamiltonian differential equation with Hamiltonian function 
\[
\Pbf(\zbf,\overline{\zbf}) = \sum_{m+m'=0}^\infty \sum_{\substack{\kbf^1+\dots+\kbf^m\\-\lbf^1-\dots-\lbf^{m'} \in\res}} \sum_{k\in\ind^m,\,l\in\ind^{m'}} P_{k,l} z_{k^1}^{\kbf^1}\dotsm z_{k^m}^{\kbf^m}\overline{z}_{l^1}^{\lbf^1}\dotsm\overline{z}_{l^{m'}}^{\lbf^{m'}},
\]
compare \eqref{eq-P}. 

The splitting structure of the modulation system for the modulation functions $\zbf$ is revealed in the following lemma: advancing the modulation functions by $\de h$ corresponds, up to a small defect, to solving Hamiltonian differential equations with Hamiltonian functions $\Hbf_0$ and $\Pbf$ one after another.

\begin{lemma}\label{lemma-splittingmfe}
We have
\[
\Ph^h_{\Hbf_0} \circ \Ph^h_{\Pbf} ( \zbf(\de t_n) ) = \widetilde{\zbf}(\de t_{n+1}) + \dbf(\de t_{n}) + \ebf(\de t_{n})
\]
with the defects $\dbf$ and $\ebf$ of \eqref{eq-defect}--\eqref{eq-defect-nearres} and where $\widetilde{z}_j^\kbf(\de t_{n+1}) = z_j^\kbf(\de t_{n+1}) \e^{-\iu(\kbf\cdot\omtbf)h}$.
\end{lemma}
\begin{proof}
Let 
\[
\bigl(\Ph^h_P(\xi)\bigr)_j = \sum_{m+m'=0}^\infty \sum_{k\in\ind^m,\,l\in\ind^{m'}} P_{j,k,l} \xi_{k^1}\dotsm\xi_{k^m} \overline{\xi}_{l^1}\dotsm\overline{\xi}_{l^{m'}}
\]
be the expansion of the flow $\Ph^h_P$ given by \eqref{eq-splittingstructure}. Then one verifies that the flow $\Ph^h_\Pbf$ is given by the same coefficients $P_{j,k,l}$,
\[
\bigl(\Ph^h_{\Pbf}(\zbf)\bigr)_j^\kbf = \sum_{m+m'=0}^\infty \sum_{k\in\ind^m,\,l\in\ind^{m'}} \sum_{\substack{\kbf^1+\dots+\kbf^m\\-\lbf^1-\dots-\lbf^{m'} =[\kbf]}} P_{j,k,l} z_{k^1}^{\kbf^1}\dotsm z_{k^m}^{\kbf^m} \overline{z}_{l^1}^{\lbf^1}\dotsm\overline{z}_{l^{m'}}^{\lbf^{m'}}.
\]
This implies that also the coefficients of the expansions of $\Ph^h_{H_0}\circ\Ph^h_P(\xi)$ and $\Ph^h_{\Hbf_0}\circ\Ph^h_{\Pbf}(\zbf)$ coincide. The coefficients in the expansion of $\Ph^h_{H_0}\circ\Ph^h_P(\xi)$ are given by \eqref{eq-split-xi}, and they also appear in the expansion \eqref{eq-defect} of $\widetilde{\zbf}(\de t_{n+1}) + \dbf(\de t_{n}) + \ebf(\de t_{n})$. The statement of the lemma follows.
\end{proof}

\subsection{Discrete almost-invariants}\label{subsec-invariants}

An essential property of modulated Fourier expansions is the existence of formal invariants. These invariants will finally allow us to consider long time intervals by patching together many of the short time intervals considered so far. They take the form
\begin{equation}\label{eq-Ic}
\Ic_m(\zbf) = \sum_{\kbf} \sum_{j\in\ind} \sum_{\substack{l\in\ind\\ n(l) = m}} k_l \abs{z_j^\kbf}^2 \myfor m\in\Nc := \{\, n(j) : j\in\ind \,\}.
\end{equation}
This is well defined (recall that the $\sum_{\kbf}$ stands for the sum over the equivalence classes $[\kbf]\in\Z^\ind/\res$) since $\sum_{l : n(l) = m} k_l = 0$ for $\kbf\in\res$ by part (c) of Assumption~\ref{assum-nonres}.

\begin{lemma}\label{lemma-invariant}
We have
\[
\Ic_m\bigl(\Ph^h_{\Hbf_0} \circ \Ph^h_{\Pbf} (\zbf(\de t_n) )\bigr) = \Ic_m\bigl(\zbf(\de t_n)\bigr) \myfor m\in\Nc.
\]
\end{lemma}
\begin{proof}
Let $\Sbf(\theta)$ be defined by
\[
\bigl(\Sbf(\theta)\zbf\bigr)_j^\kbf = \e^{\iu \theta \sum_{l : n(l) = m} k_l} z_j^\kbf
\]
for $m\in\Nc$. The Hamiltonian function $\Pbf$ from the previous subsection is invariant under the transformations $\Sbf(\theta)$, and this leads to conserved quantities by Noether's theorem: We have along a solution $\zbf=\Ph^t_{\Pbf}\zbf^0$ of the Hamiltonian differential equation with Hamiltonian function $\Pbf$
\begin{align*}
0 = \left.\frac{\dd}{\dd\theta}\right|_{\theta=0} \!\!\!\! \Pbf\bigl(\Sbf(\theta)\zbf,\overline{\Sbf(\theta)\zbf}\bigr) &= -\iu \sum_{\kbf} \sum_{j\in\ind} \sum_{\substack{l\in\ind\\ n(l) = m}} k_l \biggl(\overline{z}_{j}^{\kbf} \frac{\partial \Pbf}{\partial \overline{z}_j^\kbf} (\zbf,\overline{\zbf}) - z_{j}^{\kbf} \frac{\partial \Pbf}{\partial z_j^\kbf} (\zbf,\overline{\zbf})\biggr)\\
 &= \sum_{\kbf} \sum_{j\in\ind} \sum_{\substack{l\in\ind\\ n(l) = m}} k_l \frac{\dd}{\dd t} \abs{z_j^\kbf}^2 = \frac{\dd}{\dd t} \Ic_m(\zbf).
\end{align*}
This implies conservation of $\Ic_m$ along the flow of $\Pbf$,
\[
\Ic_m\bigl(\Ph^h_{\Pbf} (\zbf(\de t_n) ) \bigr) = \Ic_m\bigl(\zbf(\de t_n)\bigr).
\]
In the same way, one shows conservation of $\Ic_m$ along the flow of $\Hbf_0$, and the statement of the lemma follows.
\end{proof}

In the end, we are interested more in $\zbf(\de t_{n+1})$ than in $\Ph^h_{\Hbf_0} \circ \Ph^h_{\Pbf} (\zbf(\de t_n))$. The following lemma shows that $\Ic_m$ is an almost-invariant along the modulation functions~$\zbf$. 

\begin{proposition}\label{prop-almostinvariants}
We have for $\de\le\de_0$ and $s\ge s_0$
\[
\sum_{m\in\Nc} \max(1,m)^{s} \absbig{\Ic_m(\zbf(\de t_n)) - \Ic_m(\zbf(0))} \le C \ep^2 \de^{N} \myfor 0\le t_n=nh \le\de^{-1}
\]
with constants $C$, $s_0$ and $\de_0$.
\end{proposition}
\begin{proof}
Throughout the proof, we work with the representative $\kbf$ of $[\kbf]$ for which the minimum in the definition~\eqref{eq-Ga-prod} of $\Ga^\kbf$ is attained.

(a) By Lemma~\ref{lemma-splittingmfe} and Lemma~\ref{lemma-invariant} we have
\[
\absbig{\Ic_m(\zbf(\de t_{n+1})) - \Ic_m(\zbf(\de t_{n}))} \le 2 \sum_{\kbf} \sum_{j\in\ind} \sum_{\substack{l\in\ind\\ n(l) = m}} \abs{k_l} \bigl(\abs{z_j^\kbf}\abs{d_j^\kbf} + \abs{d_j^\kbf}^2 + \abs{e_j^\kbf}^2 \bigr)
\]
since $z_j^\kbf=0$ if $e_j^\kbf\ne 0$. Here, the modulation functions $z_j^\kbf$ on the right-hand side are evaluated at time $\ta=\de t_{n+1}$ and the defects $d_j^\kbf$ and $e_j^\kbf$ at time $\ta=\de t_n$.

(b) Let $\kbf\ne\mathbf{0}$ and $j=j(\kbf)=\sum_{l\in\ind} k_l l \bmod{2M} \in\ind$, and let $\bar{l}\in\ind$ be the index of largest norm $\abs{\cdot}$ with $k_{\bar{l}}\ne 0$. Then we have
\[
\abs{\bar{l}}^2 \le \abs{j} \cdot \Ga^\kbf
\]
Indeed, if $\abs{\bar{l}}^2 > \Ga^\kbf$, then necessarily $\abs{k_{\bar{l}}}=1$ and $\norm{\kbf} \cdot \abs{l}\le \abs{\bar{l}}$ for all $l\ne\bar{l}$ with $k_l\ne 0$, and hence 
\[
\abs{\bar{l}} \le \absBig{ j - \sum_{\bar{l}\ne l\in\ind} k_l l} \le \abs{j} + \frac{\norm{\kbf}-1}{\norm{\kbf}} \, \abs{\bar{l}},
\]
i.e., $\abs{\bar{l}} \le \norm{\kbf} \cdot\abs{j}$.
This implies for $s\ge 2\widehat{s}$ that 
\[
\sum_{l\in\ind} \abs{k_l} \abs{l}^{2s} \le \norm{\kbf} \cdot \abs{\bar{l}}^{2s} \le 
 \abs{j}^{2\widehat{s}} \bigl(\Ga^\kbf\bigr)^{2(s-\widehat{s})}.
\]
The last estimate improves for near-resonant indices (for which $e_j^\kbf\ne 0$) by the non-resonance condition in Assumption~\ref{assum-nonres} to
\[
\sum_{l\in\ind} \abs{k_l} \abs{l}^{2s} \le c_2^{s-2\widehat{s}} \abs{j}^{2\widehat{s}} \bigl(\Ga^\kbf\bigr)^{2(s-\widehat{s})} \de^{N}
\]
if $s-2\widehat{s} \ge s_2$.

(c) By (a), (b), the Cauchy-Schwarz inequality 
and Lemma~\ref{lemma-nj} below we have
\begin{multline*}
\sum_{m\in\Nc} \max(1,m)^{s} \absbig{\Ic_m(\zbf(\de t_{n+1})) - \Ic_m(\zbf(\de t_{n}))}\\ \le C \normv{\Gabf^{s-\widehat{s}}\zbf}_{\widehat{s}} \normv{\Gabf^{s-\widehat{s}}\dbf}_{\widehat{s}}
 + C \normv{\Gabf^{s-\widehat{s}}\dbf}_{\widehat{s}}^2 + C c_2^{s-2\widehat{s}} \de^{N} \normv{\Gabf^{s-\widehat{s}}\ebf}_{\widehat{s}}^2.
\end{multline*}
The statement of the proposition thus follows from Lemma~\ref{lemma-size} and Lemma~\ref{lemma-defect} by summing up.
\end{proof}

\begin{lemma}\label{lemma-nj}
We have
\[
\sfrac1C \abs{j}^2 \le \max \bigl(1,\abs{n(j)} \bigr) \le C \abs{j}^2 \forall j\in\ind
\]
with a positive constant $C$ depending only on $\ell$.
\end{lemma}
\begin{proof}
We have $-\abs{\ell}^2 \le n(j)\le \abs{j}^2$ since $\abs{\ell\pm j\bmod{2K}}\le \abs{\ell\pm j}$, and hence $\abs{n(j)} \le C\abs{j}^2$. To get a lower bound for $\abs{n(j)}$ we note that $\abs{\ell_1\pm j_1\bmod{2K}} \ge \min(\abs{\ell_1+ j_1},\abs{\ell_1- j_1})$, and hence
\[
\sfrac12 \abs{\ell_1+j_1\bmod{2K}}^2 + \sfrac12 \abs{\ell_1-j_1\bmod{2K}}^2 - \ell_1^2 \ge j_1^2 - 2\abs{j_1} \cdot\abs{\ell_1}.
\]
This holds not only for the first component, and we get by summing up all components
\[
n(j) \ge \abs{j}^2 - 2\abs{\ell} \cdot\abs{j}.
\]
We therefore get $n(j)\ge \frac12 \abs{j}^2$ for $4\abs{\ell} < \abs{j}$. For $4\abs{\ell}\ge \abs{j}$ we have $\max(1,\abs{n(j)})\ge 1 \ge \frac1C \abs{j}^2$. 
\end{proof}

Next we show that the almost-invariants $\Ic_m$ of \eqref{eq-Ic} are close to the \emph{super-actions}
\begin{equation}\label{eq-superactions}
I_m(\xi) = \sum_{\substack{l\in\ind\\ n(l) = m}} \abs{\xi_l}^2, \qquad m\in\Nc,
\end{equation}
that collect those \emph{actions} $\abs{\xi_l}^2$ with the same value $n(l)$.

\begin{proposition}\label{prop-actionsinvariants}
We have for $\de\le\de_0$ and $s\ge s_0$
\[
\sum_{m\in\Nc} \max(1,m)^{s} \absbig{\Ic_m(\zbf(\de t_n)) - I_m(\xi^n) } \le C \ep^2 \de \myfor 0\le t_n=nh \le\de^{-1}
\]
with constants $C$, $s_0$ and $\de_0$.
\end{proposition}
\begin{proof}
We omit the argument $\de t_n$ of the modulation functions. We have by \eqref{eq-construction-singlewavelargek}
\[
\Ic_m(\zbf) - \sum_{\substack{l\in\ind :\\ n(l)=m}} \abs{z_l^{\skla{l}}}^2 = \sum_{\substack{l\in\ind :\\ n(l)=m}} \sum_{\kbf\ne\skla{l}} k_l \abs{z_{j(\kbf)}^\kbf}^2,
\]
and part (b) of the proof of Proposition~\ref{prop-almostinvariants} together with Lemma~\ref{lemma-size}, Lemma~\ref{lemma-nj} and \eqref{eq-p0} implies
\[
\sum_{m\in\Nc} \max(1,m)^{s} \absBig{\Ic_m(\zbf) - \sum_{\substack{l\in\ind :\\ n(l)=m}} \abs{z_l^{\skla{l}}}^2} \le C \ep^2 \de^2.
\]
On the other hand, we have for the modulated Fourier expansions $\xi(t)$ \eqref{eq-mfe}
\[
\absBig{I_m(\xi(t_n)) - \sum_{\substack{l\in\ind :\\ n(l)=m}} \abs{z_l^{\skla{l}}}^2} \le 2 \sum_{\substack{l\in\ind :\\ n(l)=m}}  \Bigl(\sum_{\kbf\ne\skla{l}} \abs{z^\kbf_l} \Bigr) \Bigl( \sum_{\kbf} \abs{z^\kbf_l} \Bigr),
\]
and hence by the Cauchy-Schwarz inequality together with Lemma~\ref{lemma-size}, Lemma~\ref{lemma-nj} and \eqref{eq-p0}
\[
\sum_{m\in\Nc} \max(1,m)^{s} \absBig{I_m(\xi(t_n)) - \sum_{\substack{l\in\ind :\\ n(l)=m}} \abs{z_l^{\skla{l}}}^2} \le C \ep^2 \de.
\]
Finally, we have by Proposition~\ref{prop-error} and Lemma~\ref{lemma-nj} for the numerical solution $\xi^n$
\[
\sum_{m\in\Nc} \max(1,m)^{s} \absbig{I_m(\xi(t_n)) - I_m(\xi^n)} \le C \ep^2 \de^{N},
\]
where we have used that $\abs{\abs{\xi_j(t_n)}^2 - \abs{\xi^n_j}^2} \le \abs{\xi_j(t_n)-\xi^n_j} (\abs{\xi_j(t_n)}+\abs{\xi^n_j})$. 
Putting all this together proves the statement of the proposition.
\end{proof}

\subsection{Modulated Fourier expansion on another time interval}

All estimates of the previous subsections are valid on the time interval $0\le t_n = nh \le\de^{-1}$. We assume without loss of generality that
\begin{equation}\label{eq-ntilde}
\de^{-1} = \widetilde{n}h
\end{equation}
for some $\widetilde{n}\in\N$. In this subsection, we consider consecutive short time intervals
\[
\nu\de^{-1} \le t_n = nh \le (\nu+1)\de^{-1} \myfor \nu=0,1,2,\ldots.
\]
In principle, we can repeat the construction of a modulated Fourier expansion described in Subsects.~\ref{subsec-modequations}--\ref{subsec-constr} on these time intervals, taking $\xi^{\nu\widetilde{n}}$ as initial value instead of $\xi^0$. This gives us modulation functions $\zbf^\nu$ on the $\nu$th time interval constructed in such a way that
\[
\xi_j^n \approx \xi_j^\nu(t_n) = \sum_{\kbf} z_j^{\kbf,\nu}(\de t_n) \e^{-\iu(\kbf\cdot\omtbf) t_n} \myfor \nu \de^{-1} \le t_n\le (\nu+1)\de^{-1}.
\]

The estimates of Subsects.~\ref{subsec-size}--\ref{subsec-invariants} remain valid provided that $\xi^{\nu\widetilde{n}}$ satisfies a smallness condition as $\xi^0$ \eqref{eq-smallxi0}. In the following lemma, we bound the difference of the modulated Fourier expansions $\zbf^\nu$ and $\zbf^{\nu-1}$ at the interface $\de t_{\nu\widetilde{n}}$ of their intervals of validity.

\begin{lemma}\label{lemma-interface}
Assume that for $\nu\ge 1$
\[
\norm{\xi^{(\nu-1)\widetilde{n}}}_s \le 2\widehat{c}^{-1} \ep \myand \norm{\xi^{\nu\widetilde{n}}}_s \le 2\widehat{c}^{-1} \ep.
\]
Then we have for $\de\le\de_0$ and $s\ge s_0$
\[
\normv{\Gabf^{s-\widehat{s}}\zbf^\nu(\de t_{\nu\widetilde{n}}) - \Gabf^{s-\widehat{s}}\zbf^{\nu-1}(\de t_{\nu\widetilde{n}})}_{\widehat{s}} \le C \ep \de^{N}
\]
with constants $C$, $s_0$ and $\de_0$.
\end{lemma}
\begin{proof}
Let $n = \nu\widetilde{n}$. 

(a) We first show by induction on $q=0,\dots,N$ that
\begin{equation}\label{eq-auxinduction}
M_{q}^\nu(\zbf) := \normvBig{ \sum_{p=0}^{q} \ep \de^p \bigl(\zbf_p^\nu(\de t_n) - \zbf_p^{\nu-1}(\de t_n) \bigr) }_{s} \le C \ep \de^{q},
\end{equation}
where $\zbf_p^\nu = (z_{j,p}^{\kbf,\nu})_{j\in\ind,\kbf\in\Z^\ind}$. For this purpose we split the modulation functions, $\zbf_p=\abf_p+\bbf_p$ with $a_{j,p}^\jvec = z_{j,p}^\jvec$ and $b_{j,p}^\kbf=z_{j,p}^\kbf$ for $\kbf\ne\jvec$. 

We consider the nonlinearities $F_{j,p}^{\kbf,\nu}$ and $F_{j,p}^{\kbf,\nu-1}$ on the right-hand sides of \eqref{eq-modsys-eqmot}. By Lemma~\ref{lemma-nonlinearity} and Lemma~\ref{lemma-size} we have
\[
M_{q}^\nu(\Fbf) \le C \ep\de^{-1} M_{q-1}^\nu(\zbf),
\]
and hence by construction \eqref{eq-construction-offdiag} and \eqref{eq-construction-diag-der}
\begin{equation}\label{eq-proofinterfaceaux}
M_{q}^\nu(\bbf) \le C \de M_{q-1}^\nu(\zbf) \myand M_{q}^\nu(\dot{\abf}) \le C \de M_{q-1}^\nu(\zbf).
\end{equation}

In order to complete the inductive proof of~\eqref{eq-auxinduction}, we need a similar estimate also for $M_{q}^\nu(\abf)$. Note that by construction \eqref{eq-construction-diag-init} of $\zbf^{\nu}$
\[
\sum_{\kbf} \sum_{p=0}^{q} \ep \de^p z_{j,p}^{\kbf,\nu}(\de t_n) \e^{-\iu(\kbf\cdot\omtbf)t_n} = \sum_{\kbf} \sum_{p=0}^{q} \ep \de^p z_{j,p}^{\kbf,\nu-1}(\de t_n) \e^{-\iu(\kbf\cdot\omtbf)t_n} + r_j^n
\]
with $\norm{r^n}_{s}\le C\ep\de^{q}$ by Proposition~\ref{prop-error} (applied on the $(\nu-1)$th interval with modulation functions $z_j^{\kbf,\nu-1}$ truncated after $p=q$ instead of $p=N$ as in~\eqref{eq-cutoff}). This shows that
\begin{equation}\label{eq-aux-interface}
\biggl( \sum_{j\in\ind} \abs{j}^{2s} \absBig{\sum_{p=0}^{q} \ep \de^p \bigl( a_{j,p}^{\jvec,\nu}(\de t_n) - a_{j,p}^{\jvec,\nu-1}(\de t_n)\bigr) }^2 \biggr)^{1/2}
 \le M_{q}^\nu(\bbf) + C\ep\de^{q},
\end{equation}
and hence by~\eqref{eq-proofinterfaceaux}
\[
M_{q}^\nu(\abf) \le C \de M_{q-1}^\nu(\zbf) + C\ep\de^{q}.
\]
This completes the proof of~\eqref{eq-auxinduction}.

(b) In order to prove the statement of the lemma, we have to consider 
\[
\widehat{M}_{q}^\nu(\zbf) := \normvBig{ \sum_{p=0}^{q} \ep \de^p \Gabf^{s-\widehat{s}} \bigl(\zbf_p^\nu(\de t_n) - \zbf_p^{\nu-1}(\de t_n) \bigr) }_{\widehat{s}}
\]
instead of $M_{q}^\nu(\zbf)$. Note that $\widehat{M}_{q}^\nu(\abf) = M_{q}^\nu(\abf)$ and that the estimates~\eqref{eq-proofinterfaceaux} transfer to $\widehat{M}_q^\nu$ by~\eqref{eq-Ga-prod}. This finishes the proof of the lemma.
\end{proof}

The following proposition bounds, in the situation of the above lemma, the difference of the almost-invariants $\Ic_m$ \eqref{eq-Ic} at the interface of two time intervals.

\begin{proposition}\label{prop-diffinvariants}
Assume that for $\nu\ge 1$
\[
\norm{\xi^{(\nu-1)\widetilde{n}}}_s \le 2\widehat{c}^{-1} \ep \myand \norm{\xi^{\nu\widetilde{n}}}_s \le 2\widehat{c}^{-1} \ep.
\]
Then we have for $\de\le\de_0$ and $s\ge s_0$
\[
\sum_{m\in\Nc} \max(1,m)^{s} \absbig{\Ic_m(\zbf^\nu(\de t_{\nu\widetilde{n}})) - \Ic_m(\zbf^{\nu-1}(\de t_{\nu\widetilde{n}})) } \le C \ep^2 \de^{N} 
\]
with constants $C$, $s_0$ and $\de_0$.
\end{proposition}
\begin{proof}
This follows using part (b) of the proof of Proposition~\ref{prop-almostinvariants}, the Cauchy-Schwarz inequality and the estimates of Lemma~\ref{lemma-size}, Lemma~\ref{lemma-nj} and Lemma~\ref{lemma-interface}.
\end{proof}

\subsection{Long-time near-conservation of super-actions}\label{subsec-longtimenearcons}

We put the results of all previous subsections together to show near-conservation of the super-actions \eqref{eq-superactions} on \emph{long} time intervals of length $\de^{-N}$.

\begin{theorem}\label{thm-superactions}
The super-actions are nearly conserved for $\de\le\de_0$ and $s\ge s_0$:
\[
\sum_{m\in\Nc} \max(1,m)^{s} \absbig{ I_m(\xi^n) - I_m(\xi^0)} \le C\ep^2\de \myfor 0\le t_n=nh\le \de^{-N}
\]
with constants $C$, $s_0$ and $\de_0$.
\end{theorem}
\begin{proof}
Let $C/5$ be the maximum of the constant of Proposition~\ref{prop-diffinvariants} and the constants that appear in Propositions~\ref{prop-almostinvariants} and~\ref{prop-actionsinvariants} if $\widehat{c}^{-1}$ in \eqref{eq-smallxi0} is replaced by $2\widehat{c}^{-1}$. 

With this constant $C$, we prove the theorem for $\nu\de^{-1} \le t_n\le (\nu+1)\de^{-1}$ and $\nu=0,1,2,\ldots$ by induction on $\nu$. 
The main observation is that for $\nu\de^{-1} \le t_n\le (\nu+1)\de^{-1}$ and with $\widetilde{n}=1/(\de h)$ as in~\eqref{eq-ntilde}
\begin{equation}\label{eq-longaux}\begin{split}
\absbig{ I_m(\xi^n) - I_m(\xi^0)} &\le \absbig{I_m(\xi^n) - \Ic_m(\zbf^\nu(\de t_n))} + \absbig{\Ic_m(\zbf^\nu(\de t_n)) - \Ic_m(\zbf^\nu(\de t_{\nu\widetilde{n}}))} \\
&\qquad\qquad + \sum_{\widetilde{\nu}=1}^\nu \absbig{ \Ic_m(\zbf^{\widetilde{\nu}}(\de t_{\widetilde{\nu}\widetilde{n}})) - \Ic_m(\zbf^{\widetilde{\nu}-1}(\de t_{\widetilde{\nu}\widetilde{n}}))} \\
&\qquad\qquad + \sum_{\widetilde{\nu}=1}^\nu \absbig{ \Ic_m(\zbf^{\widetilde{\nu}-1}(\de t_{\widetilde{\nu}\widetilde{n}})) - \Ic_m(\zbf^{\widetilde{\nu}-1}(\de t_{(\widetilde{\nu}-1)\widetilde{n}})) } \\
&\qquad\qquad + \absbig{\Ic_m(\zbf^0(0)) - I_m(\xi^0) }.
\end{split}\end{equation}
After multiplying by $\max(1,m)^{s}$ and summing over $m\in\Nc$ we can apply Propositions~\ref{prop-almostinvariants}, \ref{prop-actionsinvariants} and \ref{prop-diffinvariants} to the different terms in \eqref{eq-longaux}, since, in case $\nu\ge 1$, the induction hypothesis implies for $0\le n\le \nu\widetilde{n}$
\[
\norm{\xi^n}_s^2 \le \norm{\xi^0}_s^2 + C_1^s \sum_{m\in\Nc} \max(1,m)^{s} \absbig{ I_m(\xi^n) - I_m(\xi^0)} \le \widehat{c}^{-2} \ep^2 +  C_1^s C\ep^2\de \le 2\widehat{c}^{-2} \ep^2
\]
provided that $\de\le 1/(C_1^s C\widehat{c}^2)$ with the constant $C_1$ of Lemma~\ref{lemma-nj}.
This gives
\[
\sum_{m\in\Nc} \max(1,m)^{s} \absbig{ I_m(\xi^n) - I_m(\xi^0)} \le \sfrac{2}{5}C\ep^2 \de + \sfrac{1}{5}C(2\nu+1)\ep^2\de^{N}.
\]
The statement of the theorem follows for $\nu \le \de^{-N+1}$, i.e., $t_n\le \de^{-N}$.
\end{proof}

\subsection{Proof of Theorem~\ref{thm-main}}\label{subsec-proof1}

In order to complete the proof of Theorem~\ref{thm-main}, we go in a final step back from the new variables $\xi$ introduced in Sect.~\ref{sec-trf} to the original variables $u$, in which the split-step Fourier method \eqref{eq-splitting} is formulated. Under the conditions and with the constant $C$ of Theorem~\ref{thm-superactions}, we have for $0\le t_n=nh\le\de^{-N}$ and $\de\le 1/(C_1^s C\widehat{c}^2)$
\[
\norm{\xi^n}_s^2 \le \norm{\xi^0}_s^2 + C_1^s \sum_{m\in\Nc} \max(1,m)^{s} \absbig{ I_m(\xi^n) - I_m(\xi^0)} \le 2\widehat{c}^{-2} \ep^2
\]
with the constant $C_1$ of Lemma~\ref{lemma-nj}. By Lemma \ref{lemma-norm-equivalent}, this transfers to a statement in the original variables,
\[
\norm{\Fc_{\neg \ell}(u^m)}_s \le \widehat{C} \norm{\xi^n}_s \le \sqrt{2}\widehat{C}\widehat{c}^{-1} \ep \myfor 0\le t_n=nh\le\de^{-N},
\]
as claimed in Theorem~\ref{thm-main}.

\section{On the non-resonance condition}\label{sec-nonres}

In this section, we give the proof of Theorem~\ref{thm-main2} on a sufficient condition under which Assumptions~\ref{assum-linearstability} and \ref{assum-nonres} in Theorem~\ref{thm-main} hold. The first subsection deals with the (numerical) linear stability of Assumption~\ref{assum-linearstability}, while the remaining main part of this section is devoted to the non-resonance condition of Assumption~\ref{assum-nonres}. 

From now on, we let $\ell=0$. In this case, we have $n(j)=\abs{j}^2$, and the frequencies~\eqref{eq-freq} become
\begin{equation}\label{eq-freq3}
\om_j = \frac{\arccos\bigl(\cos(\abs{j}^2h) - h\la\rh^2\sin(\abs{j}^2h)\bigr)}{h \sgn\bigl(\sin(\abs{j}^2h) +h\la\rh^2\cos(\abs{j}^2h)\bigr)}
\end{equation}
for $j\in\ind$. We introduce the set of possible values of $n(j)$: 
\[
\Nc = \{\, n(j) : j\in\ind \,\} = \{\, \abs{j}^2 : j\in\ind \,\}.
\]

\subsection{Linear stability}\label{subsec-linearstability}

We show that Assumption~\ref{assum-linearstability} is fulfilled, for $\ell=0$,  under the conditions \eqref{eq-linearstabilityexact} and \eqref{eq-strongcfl} of Theorem~\ref{thm-main2}.

\begin{lemma}\label{lemma-stab-numericalvsanalytical}
Under the step-size restriction \eqref{eq-strongcfl} the condition \eqref{eq-linearstabilityexact} of analytical linear stability implies the condition \eqref{eq-linearstability} of numerical linear stability for $0\le\rh\le\rh_0$ with $c_1=c_1(\rh_0)$. Moreover,  
\begin{equation}\label{eq-linearstabilitystrong}
1 - h^2\rh^4 + \frac{2\la\rh^2}{\mu_n} \ge \min\bigl(\sfrac12, 1+2\la\rh_0^2\bigr) >0 \forall n\in\Nc
\end{equation}
with
\begin{equation}\label{eq-mu}
\mu_n = \frac{\sin(nh)}{h\cos(nh)} = \frac{\tan(nh)}{h}.
\end{equation}
\end{lemma}
\begin{proof}
We note that for all $n\in\Nc$
\[
0 \le \frac{1}{\tan(nh)} \le \frac{1}{\tan(h)} \le \frac{1}{h} - \frac{h}{3} 
\]
by the step-size restriction \eqref{eq-strongcfl}. This yields for $\la=+1$
\[
1-h^2\rh^4 + \frac{2\la\rh^2}{\mu_n} \ge 1-h^2\rh^4
\]
and for $\la=-1$
\[
1-h^2\rh^4 + \frac{2\la\rh^2}{\mu_n} \ge 1 - h^2\rh^4 -2h\rh^2\Bigl( \frac{1}{h} - \frac{h}{3} \Bigr) \ge (1 - 2\rh^2) \Bigl(1+\frac{h^2\rh^2}{2}\Bigr).
\]
We hence get \eqref{eq-linearstabilitystrong} for $0\le \rh\le \rh_0$ from \eqref{eq-linearstabilityexact} and~\eqref{eq-strongcfl}.

The estimate~\eqref{eq-linearstabilitystrong} together with~\eqref{eq-strongcfl} implies that there exists $c_1=c_1(\rh_0)$ such that 
\[
c_1h^2 \le \sin(nh)^2 \Bigl( 1 - h^2\rh^4 + \frac{2\la\rh^2}{\mu_n} \Bigr) = 1 - \bigl(\cos(nh) - h\la\rh^2\sin(nh)\bigr)^2.
\]
for all $n\in\Nc$, and hence \eqref{eq-linearstability} holds.
\end{proof}

\subsection{Modified frequencies}\label{subsec-freqmod}

Now, we turn to Assumption~\ref{assum-nonres}, again for $\ell=0$. We begin by constructing the modified frequencies. The reason, why we use modified frequencies in the theory developed in the present paper, is that it seems to be very hard to verify the non-resonance condition in part (b) of Assumption~\ref{assum-nonres} directly for the frequencies $\om_j$ of \eqref{eq-freq3}. For the frequencies that show up after the linearization of the nonlinear Schr\"odinger equation itself around a plane wave, however, a suitable non-resonance condition can be established, see~\cite[Lemma~2.2]{Faoua}. These frequencies are $\sqrt{\abs{j}^4 + 2\la\rh^2\abs{j}^2}$, and we therefore seek modified frequencies of a similar form.

We fix $\rh_0> 0$ with \eqref{eq-linearstabilityexact} and $h$ and $K$ with \eqref{eq-strongcfl} for some $N\ge 2$ as in Theorem~\ref{thm-main2}. The frequencies $\om_j$ of \eqref{eq-freq3} are considered henceforth as functions of $\si =\rh^2$ with $0\le \si \le \si_0 := \rh_0^2$:
\begin{equation}\label{eq-freq2}
\om_j = \om_j(\sigma) = \frac{\arccos\bigl(\cos(\abs{j}^2h) - h\la\si\sin(\abs{j}^2h)\bigr)}{h}.
\end{equation}
(Note that the step-size restriction \eqref{eq-strongcfl} together with \eqref{eq-linearstabilitystrong} ensures that the sign of $\om_j$ in \eqref{eq-freq3} is positive.)

The derivative of the frequencies $\om_j$ with respect to $\si$ is given by
\[
\frac{\dd \om_j(\si)}{\dd\si} 
= \frac{\la}{\sqrt{1-h^2\si^2 + \frac{2\la\si}{\mu_{\abs{j}^2}}}}
\]
with $\mu_{\abs{j}^2}$ from \eqref{eq-mu} which is positive for $j\in\ind$ by~\eqref{eq-strongcfl}. 
This motivates the definition 
\begin{equation}\label{eq-freqmodomt}
\omt_j = \omt_j(\si) = \abs{j}^2 - \mu_{\abs{j}^2} + \sqrt{\mu_{\abs{j}^2}^2 + 2\la\si \mu_{\abs{j}^2}}, \qquad j\in\ind,
\end{equation}
of the \emph{modified frequencies} since
we then have
\[
\frac{\dd \omt_j(\si)}{\dd\si} = \frac{\la}{\sqrt{1 + \frac{2\la\si}{\mu_{\abs{j}^2}}}} \myand \omt_j(0) = \om_j(0).
\]
This implies
\[
\frac{\dd (\om_j - \omt_j)(\si)}{\dd\si} = \frac{\la h^2\si^2}{2(1-\xi + \frac{2\la\si}{\mu_{\abs{j}^2}})^{3/2}} 
\]
for all $j\in\ind$ and some $0\le\xi=\xi_j\le h^2\si^2$, and hence
\begin{equation}\label{eq-freqmodclosefreq}
\abs{\om_j - \omt_j} \le C_2 h^2 \forall j\in\disc
\end{equation}
with $C_2=C_2(\si_0)$ by \eqref{eq-linearstabilitystrong}. The modified frequencies $\omt_j$ are hence close to the original frequencies $\om_j$ as required in part (a) of Assumption~\ref{assum-nonres}.

\subsection{Bambusi's non-resonance condition for the modified frequencies}

We study resonances among the modified frequencies $\omt_j$ of \eqref{eq-freqmodomt} derived in the previous subsection. As $\omt_j=\omt_l$ for $\abs{j}^2=\abs{l}^2$, we introduce
\begin{equation}\label{eq-freqmod}
\Om_n =\Om_n(\si) = \omt_j(\si) \myfor n\in\Nc,\, j\in\ind \with n=\abs{j}^2.
\end{equation}
We verify for these modified frequencies a non-resonance condition that has been introduced by Bambusi and is widely used in the long-time analysis of infinite dimensional Hamiltonian systems. The verification is an adaptation of \cite[Sect.~5.1]{Bambusi2006} to the present situation along with some simplifications. 

We proceed roughly as follows. The aim is to show that there are a lot of ``good'' values of $\si$ for which linear combinations of frequencies do not become small. More precisely, a value of $\si$ is considered as ``good'' if for all vectors $\kbf\in\Z^{\Nc}$ and $\lbf\in\Z^{\Nc}$ with $\norm{\kbf}\le N$ and $\norm{\lbf}\le 2$ the linear combinations
\[
\kbf\cdot\Ombf + \lbf\cdot\Ombf = \sum_{n\in\Nc} k_n\Om_n + \sum_{n\in\Nc} l_n\Om_n
\]
are bounded away from zero by a negative power of $\argmax(\kbf)$, where we denote by $\argmax(\kbf)$ the largest index $n\in\Nc$ with $k_n\ne 0$ (and set $\argmax(\kbf)=1$ for $\kbf=\mathbf{0}$). The first step is to observe that it suffices to consider 
\[
\kbf\cdot\Ombf + m
\]
with integers $m$ instead of $\kbf\cdot\Ombf + \lbf\cdot\Ombf$, the reason being that $\lbf\cdot\Ombf=\pm\Om_n\pm\Om_{n'}$ is either close to an integer by the asymptotic behaviour $\Om_n\sim n+\la\si$ of the frequencies (see Lemma~\ref{lemma-asymptotic} below) or may be absorbed into $\kbf\cdot\Ombf$. This is done in Proposition~\ref{prop-bambusi} below. Furthermore, if one excludes some values of $\si$, a bound of $\kbf\cdot\Ombf + m$ can be obtained from a bound of some derivative
\[
\frac{\dd^k(\kbf\cdot\Ombf + m)}{\dd\si^k}
\]
of $\kbf\cdot\Ombf + m$ (Lemma~\ref{lemma-Pc}). Therefore we study in Lemma~\ref{lemma-det} a matrix made up of derivatives of the frequencies $\Om_n$. This matrix is such that its inverse multiplied with the vector containing the first derivatives of $\kbf\cdot\Ombf$ is just the vector containing the nonzero entries of $\kbf$. Bounding its inverse (see Lemma~\ref{lemma-det}) thus helps to study the derivatives of $\kbf\cdot\Ombf + m$.

As in the previous subsection we fix $\rh_0> 0$ with \eqref{eq-linearstabilityexact} and $h$ and $K$ with \eqref{eq-strongcfl} for some $N\ge 2$. Let us emphasize, however, that again all constants will be independent of the discretization parameters $h$ and $K$.  We will make extensive use of the asymptotic behaviour of $\mu_n$ from~\eqref{eq-mu}  and of the modified frequencies described in the following lemma.

\begin{lemma}\label{lemma-asymptotic}
We have for $0\le\rh\le\rh_0$
\[
n \le \mu_n \le C n  \myand -\frac{C}{n} \le \Om_n - n -\la\si \le 0 \myfor n\in\Nc
\]
with a constant $C=C(\si_0)$.
\end{lemma}
\begin{proof}
The estimates of $\mu_n$ follow from $nh\le \tan(nh) \le C nh$ by~\eqref{eq-strongcfl}. For the estimates of $\Om_n$ we note that
\[
\Om_n - n - \la\si = \sqrt{\bigl(\mu_n + \la\si\bigr)^2 - \si^2} - \bigl(\mu_n + \la\si\bigr).
\]
This shows 
\[
0\ge \Om_n - n - \la\si \ge \frac{-\si_0^2}{2\mu_n\sqrt{1+\frac{2\la\si}{\mu_n}}}.
\]
The estimate~\eqref{eq-linearstabilitystrong} of Lemma~\ref{lemma-stab-numericalvsanalytical} and $\mu_n\ge n$ thus lead to the claimed lower bound of $\Om_n - n - \la\si$.
\end{proof}

Now we begin with the investigation of (integer) linear combinations $\kbf\cdot\Ombf = \sum_{n\in\Nc} k_n\Om_n$ of modified frequencies \eqref{eq-freqmod}. The following lemma will help us to control derivatives of these linear combinations with respect to $\si$, which in turn will allow us to control the linear combinations themselves.

\begin{lemma}\label{lemma-det}
Let $1\le n_1<n_2<\dots<n_M$, and let $A=(a_{kl})_{k,l=1}^M$ be the matrix with entries
\[
a_{kl} = \frac{\dd^k\Om_{n_l}(\si)}{\dd\si^k}.
\]
Then for all $k,l=1,\dots,M$ and all $0\le \si\le\si_0$
\[
\abs{a_{kl}} \le C n_l^{-k+1} \myand \norm{A^{-1}}_{\infty} \le C n_M^{2M}
\]
with a constant $C=C(M,\si_0)$.
\end{lemma}
\begin{proof}
We have 
\begin{equation}\label{eq-auxderfreq}
a_{kl} = \frac{\dd^k\Om_{n_l}(\si)}{\dd\si^k} = d_k  e_l x_l^{k-1} 
\end{equation}
with $d_1 = \la$ and $d_{k+1} = -\la(2k-1)d_{k}$ for $k\ge 1$, $e_l = 1/\sqrt{1 + 2\la\si/\mu_{n_l}}$ and $x_l = e_l^2/\mu_{n_l}$. Note that for all $k,l=1,\dots,M$
\begin{equation}\label{eq-auxbounds}
c' \le \abs{d_k}\le C', \quad c'\le e_l\le C' \myand \frac{c'}{n_l}\le x_l\le \frac{C'}{n_l}
\end{equation}
with positive constants $c'=c'(\si_0)$ and $C'=C'(M,\si_0)$ by Lemmas~\ref{lemma-stab-numericalvsanalytical} and~\ref{lemma-asymptotic}. Hence, the bound on the entry $a_{kl}$ as stated in the lemma follows from the representation \eqref{eq-auxderfreq}. 

Moreover, this representation shows that
\[
A = D V E
\]
with the diagonal matrices $D=\diag(d_k)_{k=1}^M$ and $E = \diag(e_l)_{l=1}^M$ and the Vandermonde matrix $V=(x_l^{k-1})_{k,l=1}^M$. In order to examine the inverse of $A$, we first invert $V$. Its inverse is given by
\[
V^{-1} = \Bigl(\frac{v_{ij}}{ w_i}\Bigr)_{i,j=1}^M
\]
with
\[
v_{ij} = \sum_{\substack{1\le l_1<\dots<l_{M-j}\le M\\ l_1,\dots,l_{M-j}\ne i}} (-1)^j x_{l_1}\dotsm x_{l_{M-j}} \myand w_i = \prod_{\substack{1\le l \le M\\ l\ne i}} (x_i-x_l),
\]
see for example \cite[Sect.~2.8.1]{Press2007}. Since 
\[
\abs{x_i-x_l} = x_ix_l \abs{\mu_{n_l} - \mu_{n_i}} = x_ix_l \frac{\abs{n_l - n_i}}{\cos^2(\xi h)}
\]
with $\min(n_i,n_l)\le \xi\le \max(n_i,n_l)$, 
the bounds \eqref{eq-auxbounds} and the step-size restriction \eqref{eq-strongcfl} imply
\[
\norm{V^{-1}}_{\infty} \le C n_M^{2M}, \quad \norm{D^{-1}}_{\infty}\le C \myand \norm{E^{-1}}_{\infty}\le C
\]
with $C=C(M,\si_0)$. The estimate of $\norm{A^{-1}}_{\infty}$ stated in the lemma follows.
\end{proof}

Now we consider sets of values of $\si$ for which linear combinations of modified frequencies are small. We define for vectors $\kbf\in\Z^\Nc$ and $\lbf\in\Z^\Nc$ and integers $m$ the sets
\begin{equation}\label{eq-Qc}
\Qc_{\kbf,\lbf,m}(\gamma,\alpha) = \biggl\{\, \si\in [0,\si_0] : \absbig{(\kbf+\lbf)\cdot\Ombf(\si) + m} < \frac{\gamma}{\argmax(\kbf)^{\alpha}} \,\biggr\}.
\end{equation}
We first estimate the Lebesgue measure $\abs{\cdot}$ of these sets in the case $\lbf=\mathbf{0}$. 

\begin{lemma}\label{lemma-Pc}
There exists a constant $C=C(N,\al,\si_0)$ such that for all $0<\ga\le 1$, all $\norm{\kbf}\le N$ and all $m\in\Z$ with $\norm{k}+\abs{m}\ne 0$ 
\[
\abs{\Qc_{\kbf,\mathbf{0},m}(\gamma,\alpha)} \le \frac{C \gamma^{1/M}}{\argmax(\kbf)^{\al/M-4M}},
\]
where $M$ denotes the number of nonzero entries of $\kbf$.
\end{lemma}
\begin{proof}
We fix a vector $\kbf$ with $\norm{\kbf}\le N$. We may assume $\kbf\ne\mathbf{0}$ because the statement is trivial for $\kbf=\mathbf{0}$ since $\ga\le 1$.

(a) Lemma~\ref{lemma-det} shows that there exists a constant $C=C(N,\si_0)$ such that for any $0\le \si\le \si_0$ there exists $1\le k\le M$ with 
\begin{equation}\label{eq-proofbambusiaux}
\absbigg{ \frac{\dd^k (\kbf\cdot\Ombf) (\si)}{\dd\si^k} } \ge C \argmax(\kbf)^{-2M}.
\end{equation}

(b) The function $g:[0,\si_0]\rightarrow\R,\, \si\mapsto \kbf\cdot\Om(\si) + m$ is infinitely differentiable and its first $M+1$ derivatives are uniformly bounded on $[0,\si_0]$ by a constant depending only on $\si_0$ and $N$ (Lemma~\ref{lemma-det}). The property \eqref{eq-proofbambusiaux} then enables us to apply \cite[Lemma 8.4]{Bambusi1999}, which yields the statement of the lemma.
\end{proof}

Setting 
\begin{equation}\label{eq-Qc2}
\Qc(\ga,\al) = \bigcup_{\substack{\kbf : \norm{\kbf}\le N \\ \lbf : \norm{\lbf}\le 2\\ \kbf+\lbf\ne \mathbf{0}}} \Qc_{\kbf,\lbf,0}(\ga,\al)
\end{equation}
with $\Qc_{\kbf,\lbf,0}(\ga,\al)$ from~\eqref{eq-Qc}
we can now prove the following non-resonance result in the spirit of Bambusi's non-resonance condition, see \cite[Lemma~5.7]{Bambusi2006}.

\begin{proposition}\label{prop-bambusi}
Let $\al\ge (5N)^4$. Then there exists a constant $C=C(N,\al,\si_0)$ such that for all $0<\ga\le 1$
\[
\abs{\Qc(\ga,\al)} \le C \ga^{1/(2\sqrt{\al}(N+2))}.
\]
\end{proposition}
\begin{proof}
We consider the sets $\Qc_{\kbf,\lbf,0}(\ga,\al)$ for vectors $\lbf$ with $\norm{\lbf}\le 2$. Throughout this discussion we fix $0<\ga\le 1$ and $\kbf$ with $\norm{\kbf}\le N$. 

(a) For the vector $\lbf=\mathbf{0}$ the measure of the set $\Qc_{\kbf,\lbf,0}(\ga,\al)$ can be estimated with Lemma~\ref{lemma-Pc}.

(b) For $\lbf = \pm\skla{n}$ let $c'\ge 1$ be a constant such that $\abs{\pm\Om_n+\kbf\cdot\Ombf} \ge 1$ if $n>c'\argmax(\kbf)$. This constant exists by Lemma~\ref{lemma-asymptotic} and depends on $\si_0$ and $N$. Then for $n\le c'\argmax(\kbf)$ 
\[
\Qc_{\kbf,\lbf,0}(\ga,\al)\subseteq\Qc_{\kbf+\lbf,\mathbf{0},0} \bigl((c')^\al\ga,\al\bigr),
\]
whereas $\Qc_{\kbf,\lbf,0}(\ga,\al)=\emptyset$ for $n>c'\argmax(\kbf)$.

(c) For $\lbf = \pm(\skla{n}+\skla{n'})$ let similarly be $c''=c''(N,\si_0)\ge 1$ be a constant such that $\abs{\pm(\Om_n+\Om_{n'}) +\kbf\cdot\Ombf} \ge 1$ if $n+n'>c''\argmax(\kbf)$. Then for $n+n'\le c''\argmax(\kbf)$ 
\[
\Qc_{\kbf,\lbf,0}(\ga,\al)\subseteq\Qc_{\kbf+\lbf,\mathbf{0},0} \bigl((c'')^\al\ga,\al\bigr),
\]
whereas $\Qc_{\kbf,\lbf,0}(\ga,\al)=\emptyset$ for $n+n'>c''\argmax(\kbf)$.

(d) For $\lbf = \pm(\skla{n}-\skla{n'})$, where without loss of generality $n< n'$, note that with the constant $C$ of Lemma~\ref{lemma-asymptotic}
\[
\abs{\lbf\cdot\Ombf - m} \le \frac{2C}{n} \myfor m=\pm(n\pm n').
\]
Then for $n\ge C\argmax(\kbf)^{\sqrt{\al}}/\ga^{1/(2\sqrt{\al})}$
\[
\Qc_{\kbf,\lbf,0}(\ga,\al)\subseteq \Qc_{\kbf,\mathbf{0},m} \bigl(3\ga^{1/(2\sqrt{\al})},\sqrt{\al}\bigr),
\]
and this set is empty for $\abs{m}\ge c''' \argmax(\kbf)$ with a constant $c'''=c'''(N,\si_0)$ by Lemma~\ref{lemma-asymptotic}. On the other hand, we have for $n< C\argmax(\kbf)^{\sqrt{\al}}/\ga^{1/(2\sqrt{\al})}$
\[
\Qc_{\kbf,\lbf,0}(\ga,\al)\subseteq \Qc_{\kbf\pm\skla{n},\mp\skla{n'},0}(C^{\sqrt{\al}}\sqrt{\ga},\sqrt{\al}),
\]
a situation that is covered by (b).

The results (a)-(d) show that there exists a constant $c=c(N,\al,\si_0)$ such that
\[
\Qc(\ga,\al) \subseteq \bigcup_{\substack{\kbf : \norm{\kbf}\le N+2 \\ m\in\Z : \abs{m}< c''' \argmax(\kbf)\\ \norm{\kbf}+\abs{m}\ne 0}} \Qc_{\kbf,\mathbf{0},m} \bigl(c \ga^{1/(2\sqrt{\al})},\sqrt{\al}\bigr).
\]
Since the number of vectors $\kbf$ with $\norm{\kbf}\le N+2$ and $\argmax(\kbf)=L$ is at most $(N+3)L^{N+2}$ we have by Lemma~\ref{lemma-Pc} 
\[
\abs{\Qc(\ga,\al)} \le C \ga^{1/(2\sqrt{\al}(N+2))} \sum_{L= 1}^\infty L^{5N+11-\sqrt{\al}/(N+2)}
\]
with a constant $C = C(N,\al,\si_0)$. The choice of $\al$ ensures that $\sqrt{\al}\ge (N+2)(5N+13)$, and hence the latter sum converges and the proposition is proven.
\end{proof}

\begin{remark}[case $\ell\ne 0$]\label{rem-lne0}
For $\ell\ne 0$ but small, the frequencies $\om_j$ from~\eqref{eq-freq} are different from those for $\ell=0$ only for large $j$. For these large $j$, we have to deal with two differences. 

First, the frequencies of \eqref{eq-freq} (and also the modified frequencies) contain an additional summand $\frac12 \abs{\ell+j\bmod{2K}}^2 - \frac12 \abs{\ell-j\bmod{2K}}^2$. This is an integer and does not affect the proof of Lemmas~\ref{lemma-det} and \ref{lemma-Pc}, where also the integer summand $\abs{j}^2 - \mu_{\abs{j}^2}$ in the modified frequencies~\eqref{eq-freqmodomt} does not pose a problem. It does neither pose a problem in the proof of Proposition~\ref{prop-bambusi} since it is of order one for small $\ell$ (for part (b) and (c) of the proof) and is an integer (for part (d) of the proof). 

Second, the quantity $n(j)$ appearing in the frequencies of \eqref{eq-freq} can be different from $\abs{j}^2$. But for small $\ell$, these two quantities are of the same order, see Lemma~\ref{lemma-nj}. We therefore expect the statements of Lemmas~\ref{lemma-det} and \ref{lemma-Pc} and of Proposition~\ref{prop-bambusi} to transfer to this situation with constants depending on $\ell$. 
\end{remark}

\subsection{Proof of Theorem~\ref{thm-main2}}

We have already verified in Lemma~\ref{lemma-stab-numericalvsanalytical} that Assumption~\ref{assum-linearstability} is satisfied under the conditions~\eqref{eq-linearstabilityexact} and~\eqref{eq-strongcfl} of Theorem~\ref{thm-main2}. We have also verified in \eqref{eq-freqmodclosefreq} that the modified frequencies~\eqref{eq-freqmodomt} are close to the original frequencies as required in part (a) of Assumption~\ref{assum-nonres}. We will now prove that they satisfy the non-resonance condition in part (b) of Assumption~\ref{assum-nonres} for many values of $h$ and $\rh$. Note that, in the considered case $\ell=0$, part (c) of Assumption~\ref{assum-nonres} follows from part (b) since $\omt_j=\omt_l$ for all $j,l\in\ind$ with $n(j)=\abs{j}^2=\abs{l}^2 = n(l)$. 

Fix $\rh_0>0$ with~\eqref{eq-linearstabilityexact}, $h_0>0$ and $N$. In contrast to the previous subsection, we do not fix the time step-size $h$ anymore. We consider for all $0<h\le h_0$ the corresponding sets~\eqref{eq-Qc2} for $\si_0=\rh_0^2$ which we denote now by $\Qc_h(\ga,\al)$ to emphasize the dependence (of the modified frequencies, and hence the sets) on $h$. We set for $0<\ga\le 1$
\[
\Pc(\ga) = \bigl\{\, (h,\rh)\in[0,h_0] \times [0,\rh_0] : \rh^2 \not\in\Qc_h(\ga,\al) \,\bigr\}
\]
with $\al=(5N)^4$. As mentioned above, all $(h,\rh)\in\Pc(\ga)$ satisfy Assumption~\ref{assum-linearstability} with constant $c_1=c_1(\rh_0)$ and part (a) of Assumption~\ref{assum-nonres} with $\widehat{\ep}=C_2h^2$ and constant $C_2=C_2(\rh_0)$ provided that $K$ satisfies~\eqref{eq-strongcfl}.

We still have to show that for all $(h,\rh)\in\Pc(\ga)$ the modified frequencies $\Om_n = \Om_n(\rh^2)$ satisfy the non-resonance condition in part (b) of Assumption~\ref{assum-nonres} provided that~\eqref{eq-strongcfl} holds. For this purpose let $\kbf\in\Z^\Nc$ with $\norm{\kbf}\le N+1$. Then we have
\begin{equation}\label{eq-proofthmmain2-aux}
\frac{2}{\pi} \, \abs{\kbf\cdot\Ombf} \le \absbigg{\frac{\e^{\iu(\kbf\cdot\Ombf)h}-1}{h}} =:\de
\end{equation}
since the (strong\footnote{This is the first and only place, where we need that the right-hand side of \eqref{eq-strongcfl} is $\pi/(N+1)$ and not only $\pi/3$, say.}) step-size restriction \eqref{eq-strongcfl} ensures together with Lemma~\ref{lemma-asymptotic} that $\abs{\kbf\cdot\Ombf}h\le \pi$. Now we write 
\[
\kbf\cdot\Ombf = \pm\Om_{n_M} \pm\Om_{n_{M-1}}\pm\dots\pm\Om_{n_1}
\]
with $n_M\ge n_{M-1}\ge\dots\ge n_1$ in such a way that there is no pairwise cancellation ($M=\norm{\kbf}$). For $\de\le 1$ we have by~\eqref{eq-proofthmmain2-aux} and by Lemma~\ref{lemma-asymptotic} that $n_M\le c n_{M-1}$ with $c=c(N\rh_0)$. Moreover, the choice of the set $\Qc_h(\ga,\al)$ yields 
\[
\abs{\kbf\cdot\Ombf} \ge \ga \biggl(\frac{n_M^2}{c\, n_M n_{M-1}\dotsm n_1}\biggr)^\al .
\]
Combining this with~\eqref{eq-proofthmmain2-aux} we get
\[
\biggl(\frac{n_M^2}{\prod_{n\in\Nc} n^{\abs{k_n}}}\biggr)^\al \le \frac{c^\al \pi}{2\ga}\, \de.
\]
The non-resonance condition of Assumption~\ref{assum-nonres} thus holds for $c_2 = c_2(N,\ga,\rh_0)$, $\de_2=1$ and $s_2=\al N$. 

We finally have to estimate the Lebesgue measure of $\Pc(\ga)$. By Fubini's theorem and Proposition~\ref{prop-bambusi} we have
\[
\abs{\Pc(\ga)} = \rh_0h_0 - \int_0^{h_0} \Qc_h(\ga,\al)\, dh \ge \rh_0h_0 - C h_0 \ga^{1/(2\sqrt{\al}(N+2))}
\]
because the constant in this proposition is independent of $h$. The proof of Theorem~\ref{thm-main2} is thus complete if we redefine $\ga$.


\end{document}